\documentclass[a4paper,10pt,final]{article}

\usepackage{graphicx}
\usepackage{amssymb,amsmath,amsthm,mathrsfs,xfrac}
\usepackage{enumerate}  %, color
\usepackage{a4wide}
\usepackage{hyperref}
\usepackage{array}

\numberwithin{equation}{section}
\allowdisplaybreaks[4]

\newtheorem{proposition}{Proposition}[section]
\newtheorem{theorem}[proposition]{Theorem}
\newtheorem{lemma}[proposition]{Lemma}
\newtheorem{definition}[proposition]{Definition}

\newtheorem{remark}[proposition]{Remark}

\renewenvironment{proof}{\smallskip\noindent\emph{\textbf{Proof.}}%
  \hspace{1pt}}{\hspace{-5pt}{\nobreak\quad\nobreak\hfill\nobreak%
    $\square$\vspace{2pt}\par}\smallskip\goodbreak}

\newenvironment{proofof}[1]{\smallskip\noindent{\textbf{Proof~of~#1.}}%
  \hspace{1pt}}{\hspace{-5pt}{\nobreak\quad\nobreak\hfill\nobreak%
    $\square$\vspace{2pt}\par}\smallskip\goodbreak}

\newcommand{\Id}{\mathinner{\mathrm{Id}}}

\newcommand{\C}[1]{\mathbf{C}^{#1}}
\newcommand{\Cc}[1]{\mathbf{C}_c^{#1}}

\newcommand{\BV}{\mathbf{BV}}

\renewcommand{\L}[1]{{\mathbf{L}^#1}}
\newcommand{\Lloc}[1]{{\mathbf{L}_{\mathbf{loc}}^{#1}}}

\newcommand{\modulo}[1]{{\left|#1\right|}}
\newcommand{\norma}[1]{{\left\|#1\right\|}}
\newcommand{\caratt}[1]{{\chi_{\strut#1}}}
\newcommand{\reali}{{\mathbb{R}}}
\newcommand{\naturali}{{\mathbb{N}}}

\renewcommand{\epsilon}{\varepsilon}
\renewcommand{\phi}{\varphi}
\renewcommand{\theta}{\vartheta}
\renewcommand{\O}{\mathcal{O}(1)}
\newcommand{\tv}{\mathinner{\rm TV}}
\newcommand{\spt}{\mathop{\rm spt}}

\renewcommand{\d}[1]{\mathinner{\mathrm{d}{#1}}}
\renewcommand{\div}{\mathinner{\mathop{{\rm div}}}}

\newcommand{\Caption}[1]{\caption{{\small#1}}}

\newcommand{\sOmega}{\mathop{\smash[b]{\,*_{{\strut\Omega}\!}}}}
\newcommand{\sOmegaF}{\mathop{\,*_{{\strut\Omega}\!}}}

\DeclareMathOperator*{\esssup}{ess\,sup}
\DeclareMathOperator*{\essinf}{ess\,inf\vphantom{p}}

\makeatletter
\let\@fnsymbol\@arabic
\makeatother

\title{Non Local Conservation Laws in Bounded Domains}

\author{Rinaldo M.~Colombo\\ INdAM Unit, University of Brescia \\
  Italy \and Elena Rossi \\ Inria, Sophia Antipolis --  M\'editerran\'ee \\
  France}
\date{}

\begin{document}

\maketitle

\begin{abstract}

  \noindent The well posedness for a class of non local systems of
  conservation laws in a bounded domain is proved and various
  stability estimates are provided. This construction is motivated by
  the modelling of crowd dynamics, which also leads to define a non
  local operator adapted to the presence of a boundary. Numerical
  integrations show that the resulting model provides qualitatively
  reasonable solutions.

  \medskip

  \noindent\textit{2000~Mathematics Subject Classification:} 35L65,
  90B20

  \medskip


\end{abstract}

\section{Introduction}
\label{sec:Intro}

Non local conservation laws are being developed to model various
phenomena, such as the dynamics of crowd,
see~\cite{ColomboGaravelloMercier, ColomboHertyMercier,
  ColomboLecureuxPerDafermos}; vehicular
traffic,~\cite{ChalonsGoatinVillada, GoatinScialanga}; supply
chains,~\cite{ArmbrusterDegondRinghofer, ColomboHertyMercier};
granular materials,~\cite{AmadoriShen}; sedimentation
phenomena,~\cite{BetancourtBergerKarlsenTory, ChalonsGoatinVillada};
and vortex dynamics,~\cite{BakerLiMorlet}. Often, these models are set
in the whole space $\reali^N$, although the physics might require
their stating in domains with boundaries. Two difficulties typically
motivate this simplification: the rigorous treatment of boundaries and
boundary data in conservation laws is technically quite demanding,
see~\cite{BardosLerouxNedelec, bordo}, and the very meaning of non
local operators in the presence of a boundary is not straightforward,
see~\cite{DuHuangLefloch, GoatinScialanga} for recent different
approaches.

Furthermore, numerical methods for non local conservation laws are
typically developed in the case of the Cauchy problem, i.e., on all of
$\reali$, see~\cite{AmorimColombo, BetancourtBergerKarlsenTory,
  ChalonsGoatinVillada}, or on all $\reali^N$,
see~\cite{AggarwalColomboGoatin}. However, numerical integrations
obviously refer to bounded domains and proper boundary conditions need
to be singled out.

Below we tackle both the difficulties of a careful treatment of
boundary conditions and of a proper use of non local operators in the
presence of a boundary. While tackling these issues, we propose a
rigorous construction yielding the well posedness of a class of non
local conservation laws in bounded domains. Since the different
equations are coupled through non local operators, we obtain the well
posedness for a class of \emph{systems} of conservation laws in
\emph{any} space dimension. The present construction is motivated by
crowd dynamics and specific applications are explicitly considered.

Let $I$ be a real interval and $\Omega$ be a bounded open subset of
$\reali^N$. We describe the movement of $n$ populations, identified by
their densities (or \emph{occupancies})
$\rho \equiv (\rho^1, \ldots, \rho^n)$, through the following system
of non local conservation laws:
\begin{equation}
  \label{eq:1}
  \left\{
    \begin{array}{l@{\qquad}r@{\,}c@{\,}l}
      \partial_t \rho^i
      +
      \div \left[\rho^i \; V^i(t, x, \mathcal{J}^i \rho)\right]
      =
      0
      & (t,x)
      & \in
      & I \times \Omega
        \qquad i=1, \ldots, n
      \\
      \rho (t, \xi) = 0
      & (t, \xi)
      & \in
      & I \times \partial \Omega
      \\
      \rho (0,x) = \rho_o (x)
      & x
      & \in
      &\Omega
    \end{array}
  \right.
\end{equation}
where $\rho_o \in \L1 (\Omega; \reali^n)$ is a given initial datum and
$\mathcal{J}^i$ is a non local operator, so that by the writing in the
first equation of~\eqref{eq:1} we mean
\begin{displaymath}
  \partial_t \rho^i (t,x)
  +
  \div \left[
    \rho^i (t,x) \;
    V^i \! \! \left(t, x, \bigl(\mathcal{J}^i \rho (t)\bigr) (x)\right)
  \right]
  =
  0 \,.
\end{displaymath}
The choice of the zero boundary datum implies that no one can enter
$\Omega$ from outside. Nevertheless, the usual definition of solution
to conservation laws on domains with boundary,
see~\cite{BardosLerouxNedelec, Elena1, Vovelle}, allows that
individuals exit through the boundary.

\medskip

The next section is devoted to the statement of the well posedness
result. Section~\ref{sec:EX} deals with two specific sample
applications to crowd dynamics. Proofs are left to the final
sections~\ref{sec:TD} and~\ref{sec:Crowd}.

\section{Main Result}
\label{sec:MR}

We set $\reali_+ = \left[0, +\infty\right[$. The space dimension $N$,
the number of equations $n$ and the integer $m$ are fixed throughout,
with $N, n, m \geq 1$. We denote by $I$ the time interval $\reali_+$
or $[0,T]$, for a fixed $T>0$. Below, $B (x, \ell)$ for
$x \in \reali^N$ and $\ell > 0$ stands for the closed ball centred at
$x$ with radius $\ell$.

Given the map $V \colon I \times \Omega \times \reali^m \to \reali^N$,
where $(t,x,A) \in I \times \Omega \times \reali^m$ and
$\Omega \subset \reali^N$, we set
\begin{align*}
  \nabla_x V (t,x,A)
  = \
  & \left[
    \partial_{x_k} V_j (t,x,A)
    \right]_{\substack{j=1, \ldots, N \\ k = 1, \ldots, N}}
  \in \reali^{N \times N} \,,
  \\
  \nabla_A V (t,x,A)
  = \
  & \left[
    \partial_{A_l} V_j (t,x,A)
    \right]_{\substack{j=1, \ldots, N \\ l = 1, \ldots, m}} \in \reali^{N \times m},
  \\
  \nabla_{x,A} V (t,x,A)
  = \
  & \left[ \nabla_x V (t,x,A) \quad \nabla_A
    V (t,x,A) \right] \in \reali^{N \times (N+m)} \,,
  \\
  \norma{V (t)}_{\C2 (\Omega \times \reali^m ; \reali^N)}
  =\
  & \norma{V (t)}_{\L\infty (\Omega \times \reali^m ; \reali^N)}
    +
    \norma{\nabla_{x,A} V (t)}_{\L\infty (\Omega \times \reali^m ; \reali^{N\times (N+m)})}
  \\
  &  +
    \norma{\nabla^2_{x,A}
    V (t)}_{\L\infty (\Omega \times \reali^m ; \reali^{N\times(N+m) \times (N+m)})} \,.
\end{align*}
For $\rho \in \L\infty (\Omega; \reali^n)$, we also denote
$\tv (\rho) = \sum_{i=1}^n \tv (\rho^i)$.

We pose the following assumptions:
\begin{description}

\item[($\boldsymbol\Omega$)] $\Omega\subset \reali^N$ is non empty,
  open, connected, bounded and with $\C2$ boundary $\partial\Omega$.

\item[(V)] For $i=1, \ldots, n$,
  $V^i \in (\C0 \cap \L\infty) (I \times \Omega \times \reali^m;
  \reali^N)$; for all $t \in I$,
  $V^i (t) \in \C2 (\Omega \times \reali^m; \reali^N)$ and
  $\norma{V^i (t)}_{\C2 (\Omega \times \reali^m; \reali^N)}$ is
  bounded uniformly in $t$ and $i$, i.e., there exists a positive
  constant $\mathcal{V}$ such that
  $\norma{V^i (t)}_{\C2 (\Omega \times \reali^m; \reali^N)} \leq
  \mathcal{V}$ for all $t \in I$ and all $i=1, \ldots,n$.

\item[(J)] For $i=1, \ldots, n$,
  $\mathcal{J}^i \colon \L1 (\Omega;\reali^n) \to \C2
  (\Omega; \reali^m)$ is such that there exists a positive
  $K$ and a weakly increasing map
  $\mathcal{K} \in \Lloc\infty (\reali_+; \reali_+)$ such that
  \begin{enumerate}[\bf(J.1)]
  \item for all $r \in \L1 (\Omega;\reali^n)$,
    \begin{align*}
      \norma{\mathcal{J}^i (r)}_{\L\infty (\Omega;\reali^m)}
      \leq \
      & K \, \norma{r}_{\L1 (\Omega;\reali^n)} \, , % o \L\infty
      \\
      \norma{\nabla_x \mathcal{J}^i (r)}_{\L\infty (\Omega;\reali^{m\times N})}
      \leq \
      & K \, \norma{r}_{\L1 (\Omega;\reali^n)} \, , % o \L\infty
      \\
      \norma{\nabla_x^2 \mathcal{J}^i (r)}_{\L\infty (\Omega;\reali^{m\times N \times N})}
      \leq \
      & \mathcal{K}\left(\norma{r}_{\L1 (\Omega;\reali^n)}\right) \;
        \norma{r}_{\L1 (\Omega;\reali^n)} \, .
    \end{align*}

  \item for all $r_1, \, r_2 \in \L1 (\Omega;\reali^n)$
    \begin{align*}
      \norma{\mathcal{J}^i (r_1) - \mathcal{J}^i (r_2)}_{\L\infty (\Omega;\reali^m)}
      \leq  \
      & K \; \norma{r_1 - r_2}_{\L1 (\Omega;\reali^n)} \, ,
      \\
      \norma{\nabla_x\left(\mathcal{J}^i (r_1)
      - \mathcal{J}^i (r_2)\right)}_{\L\infty (\Omega;\reali^{m \times N})}
      \leq  \
      & \mathcal{K}\left(\norma{r_1}_{\L1 (\Omega;\reali^n)}\right) \;
        \norma{r_1 - r_2}_{\L1 (\Omega;\reali^n)} \, .
    \end{align*}
  \end{enumerate}

\end{description}

\noindent Throughout, $\O$ denotes a constant dependent only on norms
of the functions in the assumptions above, in particular it is
independent of time.

Recall that if $\Omega$ satisfies~\textbf{($\boldsymbol\Omega$)}, then
it also enjoys the \emph{interior sphere condition with radius
  $r > 0$}, in the sense that for all $\xi \in \partial\Omega$, there
exists $x \in \Omega$ such that $B (x,\ell) \subseteq \Omega$ and
$\xi \in \partial B (x,\ell)$ see~\cite[Section~6.4.2]{Evans}
and~\cite[Section~3.2]{GilbargTrudinger}.

In conservation laws, boundary conditions are enforced along the
boundary only where characteristic velocities enter the domain, so
that admissible jump discontinuities between boundary data and
boundary trace of the solution have to be selected. This is provided
by the following definition, based on \emph{regular entropy
  solutions}, see~\cite[Definition~3.3]{Elena1},
\cite[Definition~2.2]{Vovelle} and Definition~\ref{def:besol} below.

\begin{definition}
  \label{def:sol}
  A map $\rho \in \C0\left(I; \L1 (\Omega; \reali^n)\right)$ is a
  \emph{solution} to~\eqref{eq:1} whenever, setting
  $u^i (t,x) = V^i\left(t, x, \left(\mathcal{J}^i \rho (t) \right)
    (x)\right)$, for $i=1, \ldots, n$, the map $\rho^i$ is a regular
  entropy solution to
  \begin{equation}
    \label{eq:1sol}
    \left\{
      \begin{array}{l@{\qquad}r@{\,}c@{\,}l}
        \partial_t \rho^i
        +
        \div \left[\rho^i \; u^i(t, x)\right]
        =
        0
        & (t,x)
        & \in
        & I \times \Omega \,,
        \\
        \rho^i (t, \xi) = 0
        & (t, \xi)
        & \in
        & I \times \partial \Omega \,,
        \\
        \rho^i (0,x) = \rho^i_o (x)
        & x
        & \in
        &\Omega \,.
      \end{array}
    \right.
  \end{equation}
\end{definition}

We are now ready to state the main result of this paper.

\begin{theorem}
  \label{thm:fixpt}
  Let~$\boldsymbol{(\Omega)}$ hold. Fix $V$
  satisfying~\emph{\textbf{(V)}} and $\mathcal{J}$
  satisfying~\emph{\textbf{(J)}}. Then:
  \begin{enumerate}[\bfseries(1)]

  \item\label{item:eu} For any
    $\rho_o \in (\L\infty \cap \BV) (\Omega; \reali^n)$, there exists
    a unique $\rho \in \L\infty (I\times\Omega; \reali^n)$
    solving~\eqref{eq:1} in the sense of Definition~\ref{def:sol}.

  \item\label{item:bounds} For any
    $\rho_o \in (\L\infty \cap \BV) (\Omega; \reali^n)$ and for any
    $t \in I$,
    \begin{align*}
      \norma{\rho (t)}_{\L1 (\Omega;\reali^n)}
      \leq \
      & \norma{\rho_o}_{\L1 (\Omega;\reali^n)}
      \\
      \norma{\rho (t)}_{\L\infty (\Omega;\reali^n)}
      \leq \
      & \norma{\rho_o}_{\L\infty (\Omega;\reali^n)}
        \exp\left(
        t \, \mathcal{V} \left(
        1+ K \, \norma{\rho_o}_{\L1 (\Omega;\reali^n)}
        \right)
        \right)
      \\
      \tv\left(\rho (t)\right)
      \leq \
      & \exp\left(
        t \, \mathcal{V}
        \left(
        1 + K \,  \norma{\rho_o}_{\L1 (\Omega;\reali^n)}
        \right)
        \right)
      \\
      & \times
        \biggl[
        \mathcal{O} (1) n \, \norma{\rho_o}_{\L\infty (\Omega;\reali^n)}
        +
        \tv\left( \rho_o \right)
        +
        n \, t \, \norma{\rho_o}_{\L1 (\Omega;\reali^n)} \,
        \mathcal{V}
      \\
      & \quad \times
        \left(
        1 +  \norma{\rho_o}_{\L1 (\Omega;\reali^n)}
        \left(
        K
        + K^2  \norma{\rho_o}_{\L1 (\Omega;\reali^n)}
        + \mathcal{K}\!\left( \norma{\rho_o}_{\L1 (\Omega;\reali^n)}  \right)
        \right)
        \right)
        \biggr] .
    \end{align*}

  \item\label{item:liptime} For any
    $\rho_o \in (\L\infty \cap \BV) (\Omega; \reali^n)$ and for any
    $t, s \in I$,
    \begin{displaymath}
      \norma{\rho (t) - \rho (s)}_{\L1 (\Omega;\reali^n)}
      \leq
      \tv\left( \rho\left(\max\left\{t,s\right\}\right) \right)
      \modulo{t-s}.
    \end{displaymath}

  \item\label{item:lipid} For any initial data
    $\rho_o, \tilde\rho_o \in (\L\infty \cap \BV) (\Omega; \reali^n)$
    and for any $t \in I$, calling $\rho$ and $\tilde\rho$ the
    corresponding solutions to~\eqref{eq:1},
    \begin{displaymath}
      \norma{\rho (t) - \tilde\rho (t)}_{\L1 (\Omega;\reali^n)}
      \leq
      e^{\mathcal{L}(t)} \,
      \norma{\rho_o - \tilde \rho_o}_{\L1 (\Omega;\reali^n)},
    \end{displaymath}
    where $\mathcal{L}(t)>0$ depends on~$\boldsymbol{(\Omega)}$,
    \emph{\textbf{(V)}}, \emph{\textbf{(J)}} and on
    \begin{displaymath}
      R = \max\left\{
        \norma{\rho_o}_{\L1 (\Omega;\reali^n)}, \,
        \norma{\tilde \rho_o}_{\L1 (\Omega;\reali^n)}, \,
        \norma{\rho_o}_{\L\infty (\Omega;\reali^n)}, \,
        \norma{\tilde \rho_o}_{\L\infty (\Omega;\reali^n)}, \,
        \tv (\rho_o), \,
        \tv (\tilde \rho_o)
      \right\}.
    \end{displaymath}

  \item\label{item:stab} Fix
    $\rho_o \in (\L\infty \cap \BV) (\Omega; \reali^n)$. Let
    $\tilde V$ satisfy~\emph{\textbf{(V)}} with the same constant
    $\mathcal{V}$. Call $\rho$ and $\tilde \rho$ the solutions to
    problem~\eqref{eq:1} corresponding respectively to the choices $V$
    and $\tilde V$. Then, for any $t \in I$,
    \begin{displaymath}
      \norma{\rho (t) - \tilde\rho (t)}_{\L1 (\Omega;\reali^n)}
      \leq
      \mathcal{C} (t)
      \int_0^t
      \norma{V (s) - \tilde V (s)}_{\C1(\Omega \times \reali^m; \reali^{nN})} \d{s}
    \end{displaymath}
    where $\mathcal{C}$ depends on~$\boldsymbol{(\Omega)}$,
    \emph{\textbf{(V)}}, \emph{\textbf{(J)}} and on the initial datum,
    see~\eqref{eq:21}.

  \item\label{it:positivity} For $i=1, \ldots, n$, if
    $\rho^i_o \geq 0$ a.e.~in $\Omega$, then $\rho^i (t) \geq 0$
    a.e.~in $\Omega$ for all $t \in I$.

  \end{enumerate}
\end{theorem}

\noindent Section~\ref{sec:TD} is devoted to the proof of the theorem
above. Here, we underline that the total variation estimate
in~(\ref{item:bounds})~is qualitatively different from the analogous
one in the case of no boundary, see Remark~\ref{rem:note}.

\section{The Case of Crowd Dynamics}
\label{sec:EX}

The above analytic results are motivated also by their applicability
to equations describing the motion of a crowd, identified through its
time and space dependent density $\rho = \rho (t,x)$. Various
macroscopic crowd dynamics models based on non local conservation laws
were recently considered, see for
instance~\cite{ColomboGaravelloMercier, ColomboHertyMercier,
  ColomboLecureuxPerDafermos}, as well
as~\cite[Section~3.1]{AggarwalColomboGoatin}. Therein, typically, non
local interactions among individuals are described through space
convolution terms like $\rho (t) * \eta$, for a suitable averaging
kernel $\eta$. We refer to~\cite{CristianiPiccoliTosin} for a
different approach and to~\cite{BellomoPiccoliTosin} for a recent
review on the modelling of crowd dynamics.

Due to the absence of well posedness results in bounded domains, none
of the results cited above considers the presence of boundaries. On
the one hand, the choice of the crowd velocity may well encode the
presence of boundaries but, on the other hand, the visual horizon of
each individual should definitely not neglect the presence of the
boundary. With this motivation, below we introduce a non local
operator consistent with the presence of boundaries and show how the
theoretical results above allow to formulate equations where each
individual's horizon is affected by the presence of the walls.

To this aim, we use the following modification of the usual
convolution product
\begin{align}
  \label{eq:*omega}
  (\rho \sOmega \eta) (x)
  = \
  & \dfrac{1}{z (x)} \int_\Omega \rho (y) \; \eta (x-y) \d{y} \,,
    \qquad \mbox{where}
  \\
  \label{eq:z}
  z (x)
  = \
  & \int_{\Omega} \eta (x-y) \d{y} \,.
\end{align}
A reasonable assumption on the kernel $\eta$ is:

\begin{description}
\item[($\boldsymbol{\eta}$)] $\eta (x) = \tilde\eta (\norma{x})$,
  where $\tilde\eta \in \C2 (\reali_+; \reali)$,
  $\spt \tilde\eta = [0, \ell_\eta]$, where $\ell_\eta > 0$,
  $\tilde\eta' \leq 0$ and $\int_{\reali^N} \eta (\xi) \d\xi = 1$.
\end{description}

\noindent In other words, $(\rho \sOmega \eta) (x)$ is an average of
the crowd density $\rho$ in $\Omega$ around $x$. Note also that
$\rho \sOmega \eta$ is well defined by~\eqref{eq:*omega}: indeed,
under assumptions~\textbf{($\boldsymbol\Omega$)}
and~\textbf{($\boldsymbol\eta$)}, $z$ may not vanish in $\Omega$, see
Lemma~\ref{lem:z}. As a side remark, note
that~\textbf{($\boldsymbol\eta$)} ensures $\eta\geq 0$.

We investigate the properties of the non local operator defined
through~\eqref{eq:*omega}--\eqref{eq:z}.

\begin{lemma}
  \label{lem:Omega}
  Let $\Omega$ satisfy~$\boldsymbol{(\Omega)}$, $\eta$
  satisfy~$\boldsymbol{(\eta)}$ and
  $\rho \in \L\infty (\Omega; \reali_+)$. Then,
  \begin{displaymath}
    (\rho \sOmega \eta) \in \C2 (\Omega; \reali_+)
    \quad \mbox{ and }\quad
    (\rho \sOmega \eta) (x)
    \in
    [
    \essinf_{B (x,\ell_\eta) \cap \Omega} \rho,
    \esssup_{B (x,\ell_\eta) \cap \Omega} \rho
    ]
  \end{displaymath}
  so that, in particular,
  $(\rho \sOmega \eta) (\Omega) \subseteq [0, \norma{\rho}_{\L\infty
    (\Omega;\reali)}]$.
\end{lemma}

\noindent The proof is in Section~\ref{sec:Crowd}, where other
properties of the modified convolution~\eqref{eq:*omega}--\eqref{eq:z}
are proved.

As a sample of the possible applications of Theorem~\ref{thm:fixpt} to
crowd dynamics, we consider below two specific situations, where we
set $N = 2$, write $x \equiv (x_1, x_2)$ for the spatial coordinate
and denote $\partial_1 = \partial_{x_1}$,
$\partial_2 = \partial_{x_2}$,

The numerical integrations below are obtained through a suitable
adaptation of the Lax--Friedrichs method, on the basis
of~\cite{AggarwalColomboGoatin, AmorimColombo}, adapted as suggested
in~\cite[Formula~(14)]{BlandinGoatin} to reduce the effects of the
numerical viscosity.

For further results on crowd modelling, see for
instance~\cite{ColomboGaravelloMercier, HelbingEtAlii2001,
  Hoogendoorn2005147} and the references therein.

\subsection{Evacuation from a Room}

We now use~\eqref{eq:1} to describe the evacuation of a
region, say $\Omega$. To this aim, consider the equation:
\begin{equation}
  \label{eq:23}
  \partial_t \rho
  +
  \div \left[
    \rho \;\,
    v (\rho\sOmega\eta_1) \,
    \left(
      w (x)
      -
      \beta
      \dfrac{\nabla (\rho \sOmegaF\eta_2)}{\sqrt{1+\norma{\nabla (\rho\sOmega\eta_2)}^2}} \right)
  \right]
  =
  0 \,.
\end{equation}
Here, each individual adjusts her/his speed according to the average
population density around her/him, according to the function $v$,
which is $\C2$, bounded and non increasing. The velocity direction of
each individual is given by the fixed $\C2$ vector field $w$, which
essentially describes some sort of \emph{natural} path to the exit,
the exit being the portion of $\partial\Omega$ where $w$ points
outwards of $\Omega$. This direction is then adjusted by the non local
term
$- \beta \; \nabla (\rho \sOmegaF\eta_2) \left/ \sqrt{1+\norma{\nabla
      (\rho\sOmega\eta_2)}} \right.$, which describes the tendency of
avoiding regions with high (average) density gradient,
see~\cite{ColomboGaravelloMercier, ColomboLecureuxPerDafermos}.

\begin{lemma}
  \label{lem:Corridor}
  Let $\Omega$ satisfy~$\boldsymbol{(\Omega)}$. Assume that
  $v \in \C2 (\reali_+; \reali_+)$ and $w \in \C2 (\Omega;\reali^2)$
  are bounded in $\C2$. If moreover $\eta_1$, $\eta_2$
  satisfy~$\boldsymbol{(\eta)}$ with $\eta_2$ of class $\C3$, then
  equation~\eqref{eq:23} fits into~\eqref{eq:1}, \emph{\textbf{(V)}}
  and~\emph{\textbf{(J)}} hold, so that Theorem~\ref{thm:fixpt}
  applies.
\end{lemma}

\noindent The proof is deferred to Section~\ref{sec:Crowd}.

\smallskip

As a specific example we consider a square room, say $\Omega$, with a
door $D$, with $D \subseteq \partial \Omega$, and two columns each of
size $0.5 \times 0.625$, placed near to the door, symmetrically as the
grey rectangles in the figure in~\eqref{eq:25}. We also
set\\
\begin{minipage}{0.4\linewidth}
  \centering
  \includegraphics[width=\linewidth,trim = 60 10 60
  20,clip=true]{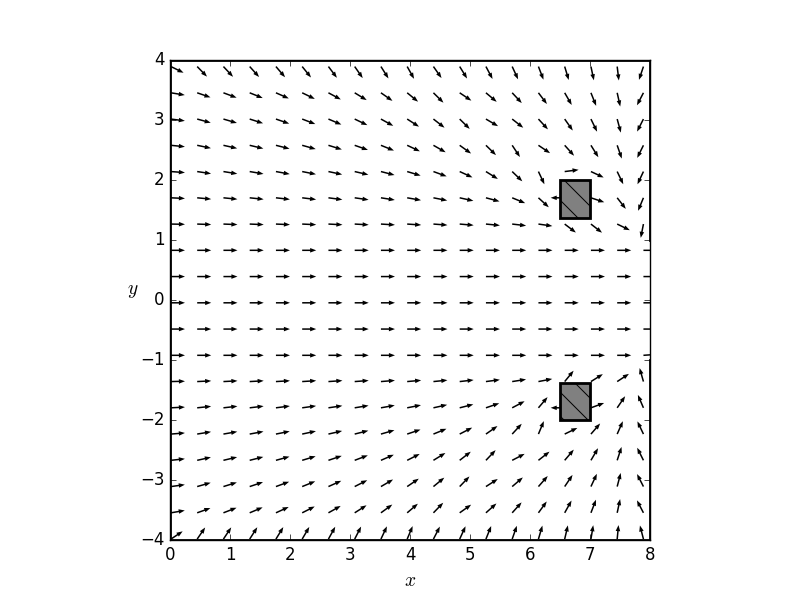}
  \label{fig:oneVF}
\end{minipage}%
\begin{minipage}{0.6\linewidth}
  \begin{equation}
    \label{eq:25}
    \begin{array}{@{}c@{}}
      \begin{array}{@{}r@{\,}c@{\,}l@{}}
        \Omega
        & =
        & [0,8] \times [-4, 4]
        \\
        D
        & =
        & \{8\} \times [-1,1]
        \\
        \tilde\eta_i (\xi)
        & =
        & \dfrac{315}{128 \, \pi \, {l_i}^{18}} \; ({l_i}^4 -
          \xi^4)^4 \; \caratt{[0,l_i]} (\xi)
        \\
        v (r)
        & =
        & 2 \; \min\left\{1, \max\left\{0,
          (1-(r/4)^3)^3\right\}\right\}
        \\
        w (x)
        & =
        & \mbox{ see the figure here on the left,}
      \end{array}
      \\
      l_1
      =
      0.625 \,,
      \quad
      l_2
      =
      1.5 \,,
      \quad
      \beta
      =
      0.6 \,.
    \end{array}
  \end{equation}
\end{minipage}\\
The vector field $w = w (x)$ is obtained as a sum of the unit vector
tangent to the geodesic from $x$ to the door and a discomfort vector
field with maximal intensity along the walls. The numerical
integration corresponding to a locally constant initial datum is
displayed in Figure~\ref{fig:one}.
\begin{figure}[!ht]
  \centering
  \includegraphics[width=0.3\linewidth,trim = 90 10 90
  20,clip=true]{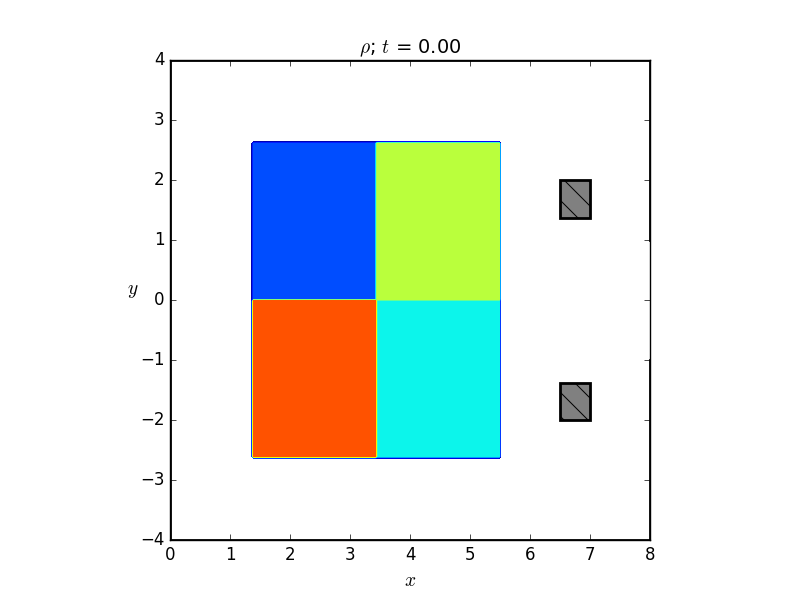}%
  \includegraphics[width=0.3\linewidth,trim = 90 10 90
  20,clip=true]{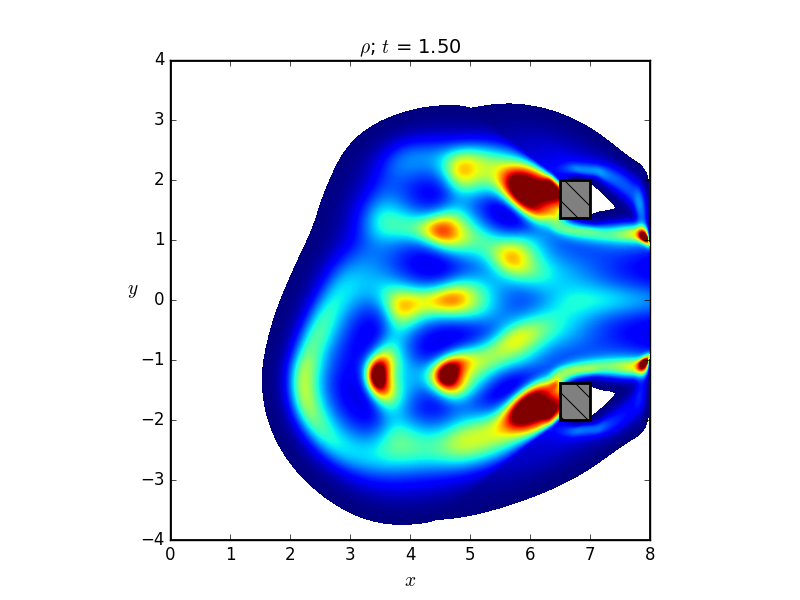}%
  \includegraphics[width=0.3\linewidth,trim = 90 10 90 20,clip=true]{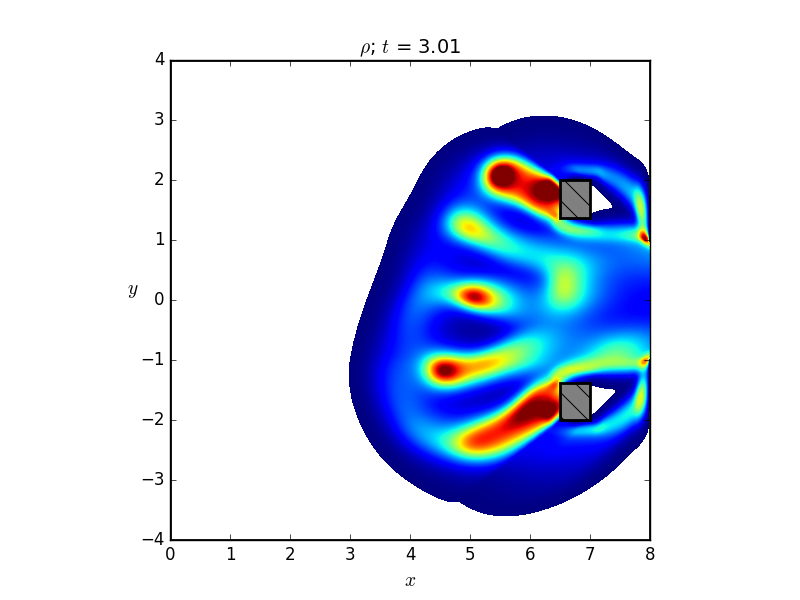}\\
  \includegraphics[width=0.3\linewidth,trim = 90 10 90
  20,clip=true]{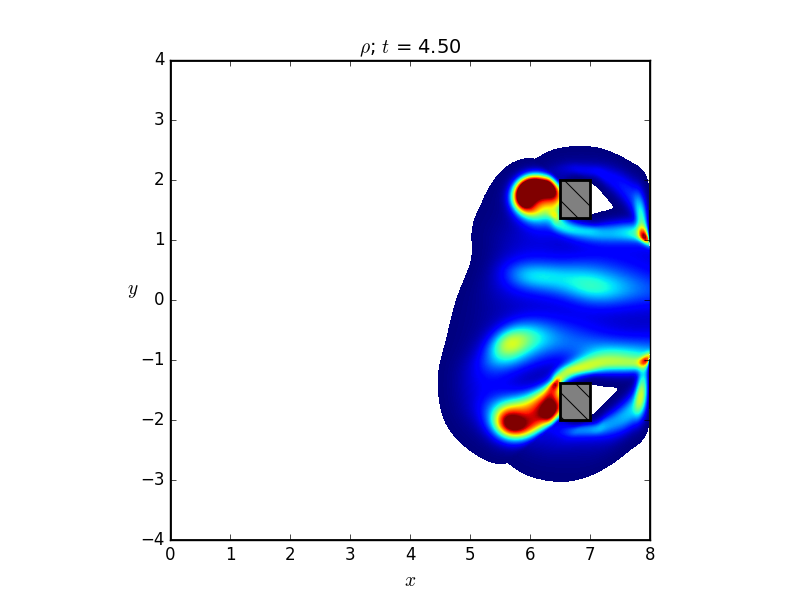}%
  \includegraphics[width=0.3\linewidth,trim = 90 10 90
  20,clip=true]{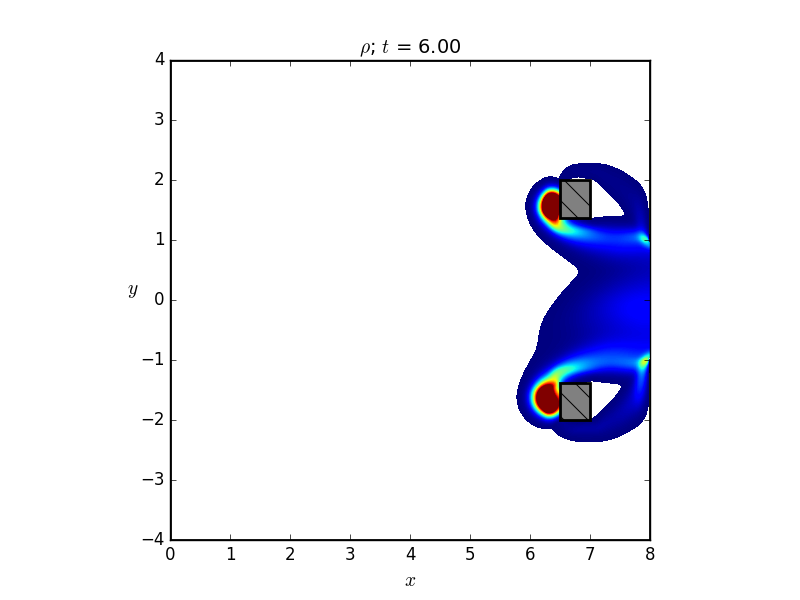}%
  \includegraphics[width=0.3\linewidth,trim = 90 10 90
  20,clip=true]{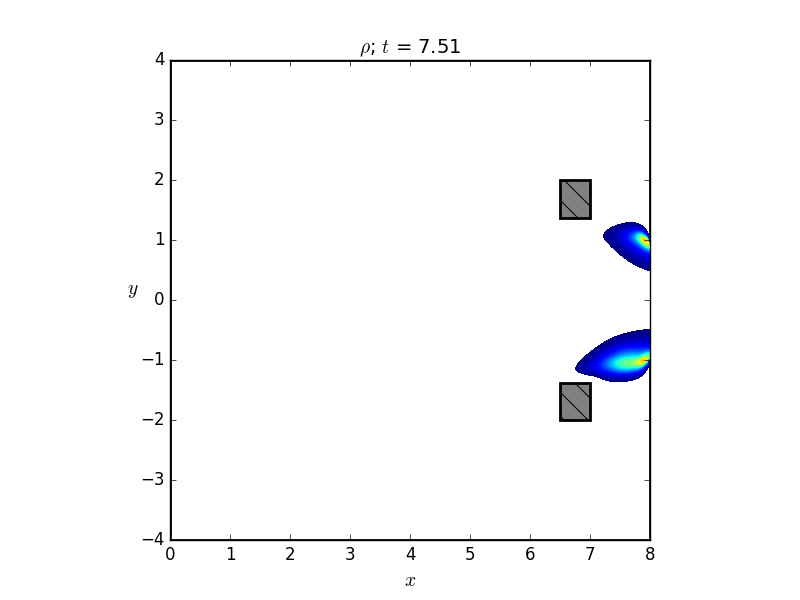}%
  \Caption{Plot of the level curves of the solution
    to~\eqref{eq:1}--\eqref{eq:23}--\eqref{eq:25}, computed
    numerically at the times $t=0,\;1.5,\; 3.0,\; 4.5,\; 6.0,\; 7.5$,
    corresponding to the initial data in the top left figure,
    consisting of $5$, $14$, $9$ and $20$ people (clockwise starting
    from the top left) in the $4$ quadrants displayed in the first
    figure. In this integration, the mesh sizes are
    $dx = dy = 0.03125$.}
  \label{fig:one}
\end{figure}
The solution displays a realistic behaviour, with queues being formed
behind the obstacles. For further details on the modelling and
numerical issues related
to~\eqref{eq:1}--\eqref{eq:23}--\eqref{eq:25}, we refer
to~\cite{prossimo}.

\subsection{Two Ways Movement along a Corridor}

The validity of Theorem~\ref{thm:fixpt} also for systems of equations
allows to consider the case of interacting populations. A case widely
considered in the literature, see for
instance~\cite{AggarwalColomboGoatin, ColomboGaravelloMercier,
  ColomboLecureuxPerDafermos, CristianiPiccoliTosin,
  HelbingEtAlii2001} and the references in~\cite{BellomoPiccoliTosin},
is that of two groups of pedestrians heading in opposite directions
along a corridor, say $\Omega$, with exits, say $D$, on each of its
sides. With the notation in Section~\ref{sec:MR}, this amounts to set
$N=2$, $n=2$ and to{\small
  \begin{equation}
    \label{eq:24}
    \!\!\!\!\!\!\!
    \left\{
      \begin{array}{@{}r@{}l@{}}
        \partial_t \rho^1
        {+}
        \div \!\! \left[
        \rho^1
        v^1 ((\rho^1{+}\rho^2) {\sOmega} \eta_1^{11}) \!
        \left(
        w^1 (x)
        -
        \dfrac{\beta_{11}
        \nabla (\rho^1 \sOmegaF\eta_2^{11})}{\sqrt{1+\norma{\nabla (\rho^1\sOmega\eta_2^{11})}^2}}
        -
        \dfrac{\beta_{12}
        \nabla (\rho^2 \sOmegaF\eta_2^{12})}{\sqrt{1+\norma{\nabla (\rho^2\sOmega\eta_2^{12})}^2}} \right)\!
        \right]_{\vphantom{|}}
        & =
          0,
        \\
        \partial_t \rho^2
        {+}
        \div \!\! \left[
        \rho^2
        v^2 ((\rho^1{+}\rho^2) {\sOmega} \eta_1^{22}) \!
        \left(
        w^2 (x)
        -
        \dfrac{\beta_{21}
        \nabla (\rho^1 \sOmegaF\eta_2^{21})}{\sqrt{1+\norma{\nabla (\rho^1\sOmega\eta_2^{21})}^2}}
        -
        \dfrac{\beta_{22}
        \nabla (\rho^2 \sOmegaF\eta_2^{22})}{\sqrt{1+\norma{\nabla (\rho^2\sOmega\eta_2^{22})}^2}} \right) \!
        \right]
        &  =
          0.
      \end{array}
    \right.
    \!\!\!\!\!\!\!\!\!\!\!\!
  \end{equation}}

\noindent The various terms in the expressions above are
straightforward extensions of their analogues in~\eqref{eq:23}. For
instance, in view of~\eqref{eq:*omega}--\eqref{eq:z},
$v^i = v^i \left((\rho^1+\rho^2)\sOmega \eta^{ii}_1\right)$ describes
how the maximal speed of the population $i$ at a point $x$ depends on
the average total density of $\rho^1+\rho^2$ in $\Omega$ around
$x$. Similarly, the term
$-\beta_{ij} \, \nabla (\rho^i \sOmega \eta^{ij}_2) \left/
  \sqrt{1+\norma{\rho^i \sOmega \eta^{ij}_2}^2}\right.$
describes the tendency of individuals of the $i$-th population to
avoid increasing values of the average density of the $j$-th
population, in the same spirit of the similar term in~\eqref{eq:23}.

\begin{lemma}
  \label{lem:TwoPop}
  Let $\Omega$ satisfy~$\boldsymbol{(\Omega)}$. Assume that
  $v^1, v^2 \in \C2 (\reali_+; \reali_+)$ and
  $w^1, w^2 \in \C2 (\Omega;\reali^2)$ are bounded in $\C2$. If
  moreover $\eta_1^{ii}$, $\eta^{ij}_2$ satisfy~$\boldsymbol({\eta)}$
  with $\eta^{ij}_2$ of class $\C3$ for $i,j=1,2$, then
  equation~\eqref{eq:24} fits into~\eqref{eq:1}, \emph{\textbf{(V)}}
  and~\emph{\textbf{(J)}} hold, so that Theorem~\ref{thm:fixpt}
  applies.
\end{lemma}

\noindent The proof is deferred to Section~\ref{sec:Crowd}.

\smallskip

A qualitative picture of the possible solutions
to~\eqref{eq:1}--\eqref{eq:24} is obtained through the following numerical
integration, corresponding to the choices\\
\begin{minipage}{0.4\linewidth}
  \centering
  \includegraphics[width=0.9\linewidth,trim = 30 10 50
  20,clip=true]{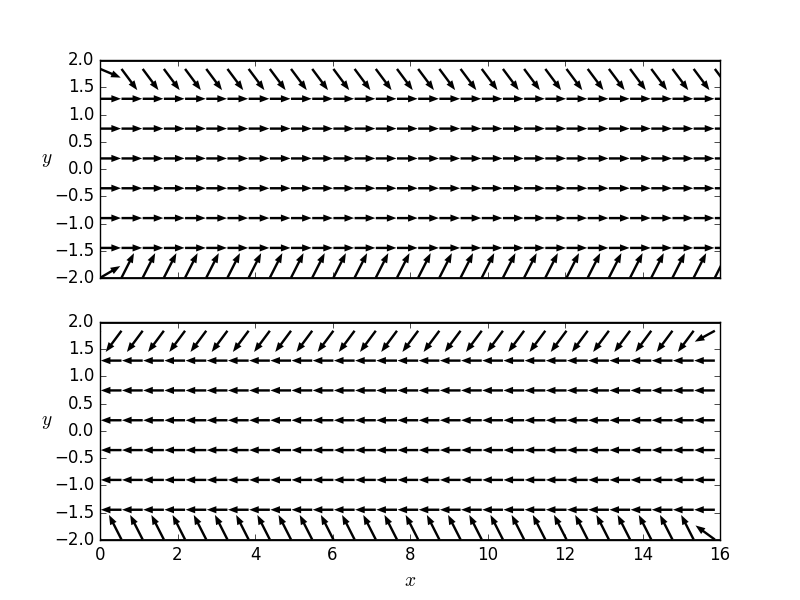}
  \label{fig:twoVF}
\end{minipage}%
\begin{minipage}{0.6\linewidth}
  \begin{equation}
    \label{eq:20}
    \!\!
    \begin{array}{@{}c@{}}
      \Omega = [0,16] \times [-2,2]
      \,,\qquad
      D = \{0,16\} \times [-2, 2]
      \\
      \begin{array}{@{}r@{\,}c@{\,}l@{}}
        \tilde\eta^{ij}_l (\xi)
        & =
        & \dfrac{315}{128 \, \pi \, {(l_i^{ij})}^{2}} \; (1 -
          (\xi/l_i^{ij})^4)^4 \; \caratt{[0,l_i^{ij}]} (\xi)
        \\
        v^1 (r)
        & =
        & \min\left\{1, \max\left\{0,
          (1-(r/4.5)^3)^3\right\}\right\}
        \\
        v^2 (r)
        & =
        & 1.5 \; \min\left\{1, \max\left\{0,
          (1-(r/4.5)^3)^3\right\}\right\}
        \\
        w^1 (x)
        & =
        & \mbox{ see the figure here on the left, top}
        \\
        w^2 (x)
        & =
        & \mbox{ see the figure here on the left, bottom}
      \end{array}
      \\
      l^{ii}_1
      =
      0.1875 \,,
      \quad
      l^{ij}_2
      =
      0.5 \,,
      \quad
      \beta_{ii}
      =
      0.2 \,,
      \quad
      \beta_{ij}
      =
      0.5 \,.
    \end{array}
  \end{equation}
\end{minipage}
\\
for $i,j =1,2$, see Figure~\ref{fig:two}. Note the complex dynamics
arising due to the formation of regions with high density.
\begin{figure}[!h]
  \centering
  \includegraphics[width=0.33\linewidth,trim = 25 10 25
  25,clip=true]{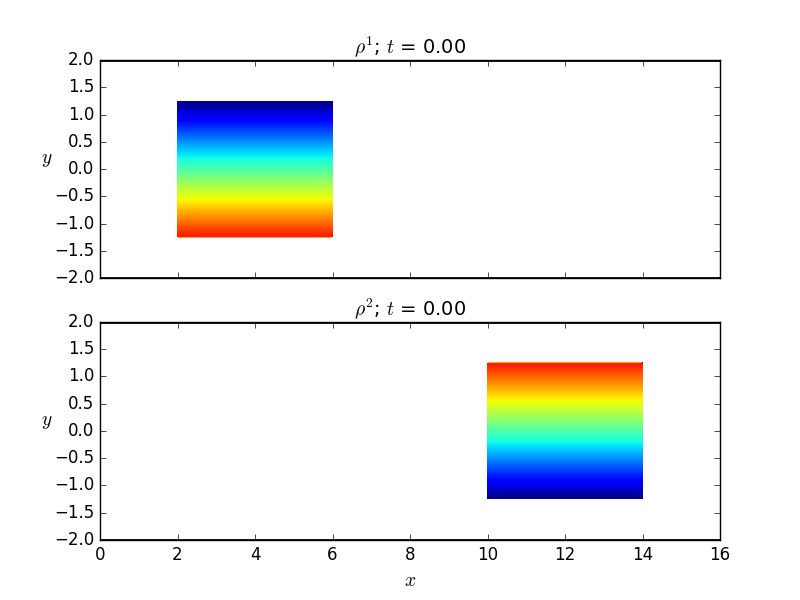}%
  \includegraphics[width=0.33\linewidth,trim = 25 10 25
  25,clip=true]{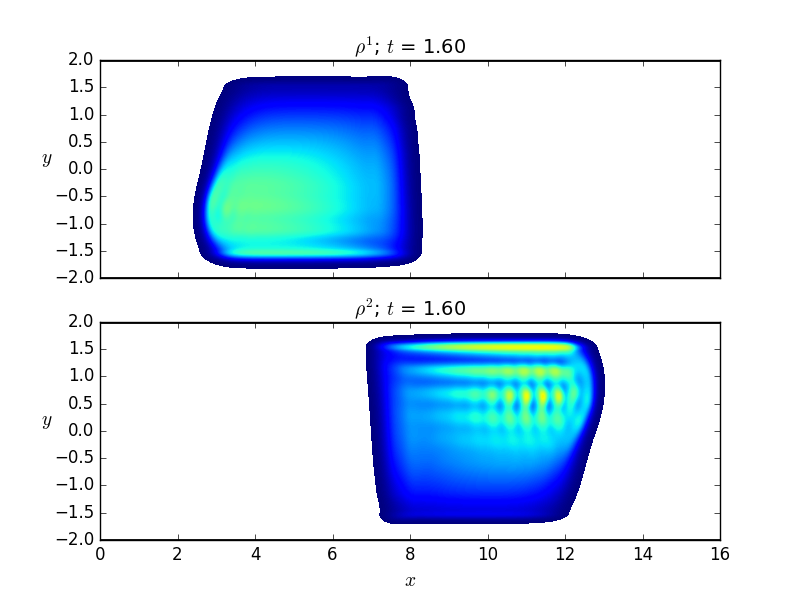}%
  \includegraphics[width=0.33\linewidth,trim = 25 10 25 25,clip=true]{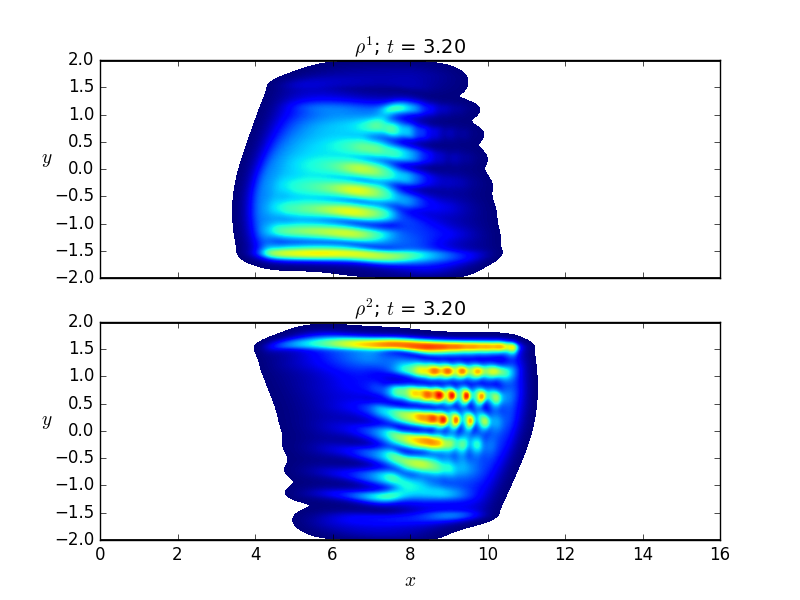}\\\  \\
  \includegraphics[width=0.33\linewidth,trim = 25 10 25
  25,clip=true]{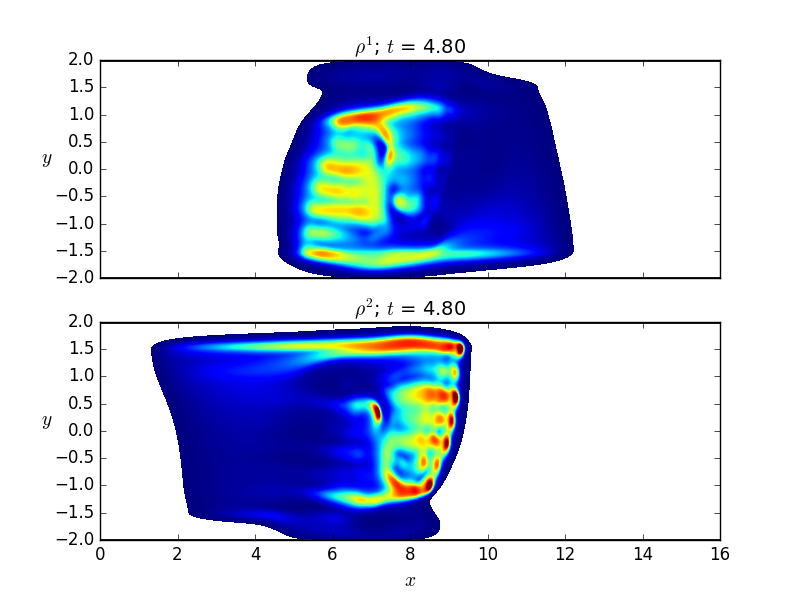}%
  \includegraphics[width=0.33\linewidth,trim = 25 10 25
  25,clip=true]{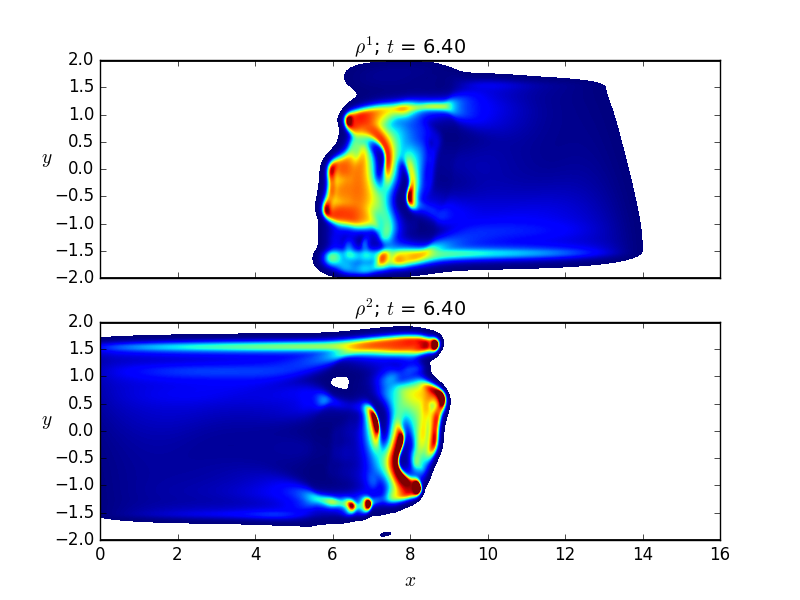}%
  \includegraphics[width=0.33\linewidth,trim = 25 10 25
  25,clip=true]{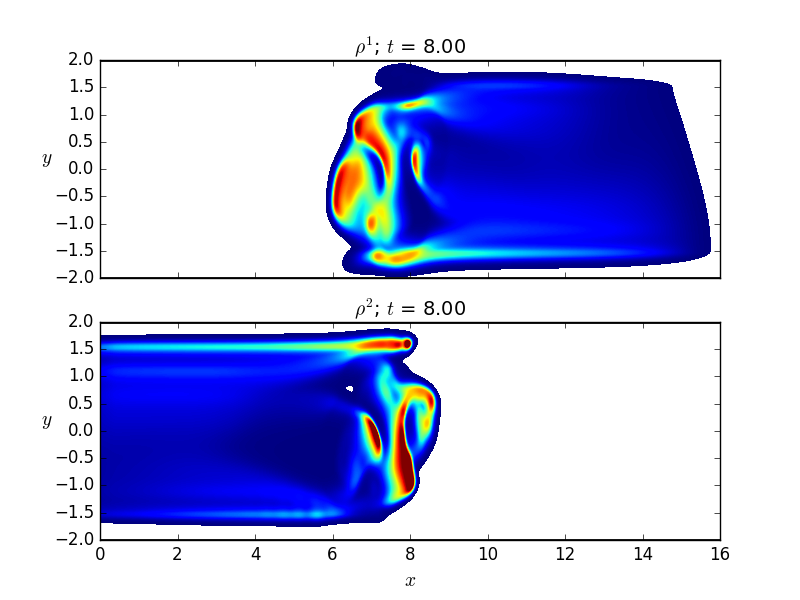}%
  \Caption{Plot of the level curves of the solution
    to~\eqref{eq:1}--\eqref{eq:24}--\eqref{eq:20}, computed
    numerically at the times $t=0,\;1.6,\; 3.2,\; 4.8,\; 6.4,\; 8.0$.
    First and third rows refer to $\rho^1$, while the second and
    fourth one to $\rho^2$. The initial datum varies linearly along
    the $y$ coordinate between $0$ and $4$. In this integration, the
    mesh sizes are $dx = dy = 0.015625$.}
  \label{fig:two}
\end{figure}
This description is consistent with the typical \emph{self
  organization} of crowd motions, see~\cite{Helbing2012,
  HelbingEtAlii2001}: queues consisting of pedestrian walking in the
same direction are formed, in particular at time $3.20$.

\section{Proofs Related to Section~\ref{sec:MR}}
\label{sec:TD}

We recall the basic properties of the following (local) IBVP
\begin{equation}
  \label{eq:2}
  \left\{
    \begin{array}{l@{\qquad}r@{\,}c@{\,}l}
      \partial_t r
      +
      \div \left[r \; u(t, x)\right]
      =
      0
      & (t,x)
      & \in
      & I \times \Omega
      \\
      r (t, \xi) = 0
      & (t, \xi)
      & \in
      & I \times \partial \Omega
      \\
      r (0,x) = r_o (x)
      & x
      & \in
      &\Omega \,,
    \end{array}
  \right.
\end{equation}
where we assume that
\begin{description}
\item[(u)] $u \colon I \times \Omega \to \reali^N$ is such that
  $u \in (\C0 \cap \L\infty) (I \times \Omega;\reali^N)$, for all
  $t \in I$, $u (t) \in \C2 (\Omega;\reali^N)$ and
  $\norma{u (t)}_{\C2 (\Omega;\reali^N)}$ is uniformly bounded in $I$.
\end{description}

\noindent We refer to~\cite{Elena1} for a comparison among various
definitions of solutions to~\eqref{eq:2}. Recall the concept of
\emph{RE--solutions}, which first requires an extension
of~\cite[Chapter~2, Definition~7.1]{MalekEtAlBook}. Note that,
although the equation in~\eqref{eq:2} is linear, jump discontinuities
may well arise between the solution and the datum assigned along the
boundary.

\begin{definition}[{\cite[Definition~2]{Vovelle}}]
  \label{def:beef}
  The pair
  $(H,Q) \in \C2 (\reali^2;\reali) \times \C2 (I \times \overline
  \Omega \times \reali^2; \reali^N)$ is called a \emph{boundary
    entropy-entropy flux pair} for the flux $f (t,x,r) = r \, u (t,x)$
  if:
  \begin{enumerate}[i)]
  \item for all $w \in \reali$ the function $z \mapsto H (z,w)$ is
    convex;
  \item for all $t \in I $, $x\in\overline\Omega$ and
    $z,\, w \in \reali$,
    $\partial_zQ (t,x,z,w) = \partial_z H (z,w) \, u (t,x)$;
  \item for all $t \in I$, $x\in\overline\Omega$ and $w \in \reali$,
    $H (w,w)= 0$, $Q (t,x,w,w) = 0$ and $\partial_z H (w,w) = 0$.
  \end{enumerate}
\end{definition}

\noindent Note that if $H$ is as above, then $H \geq 0$.

\begin{definition}[{\cite[Definition~3.3]{Elena1}}]
  \label{def:besol}
  A \emph{Regular Entropy solution} (RE-solution) to the
  initial--boundary value problem~\eqref{eq:2} on $I$ is a map
  $r \in \L\infty\left(I \times \Omega; \reali \right)$ such that for
  any boundary entropy-entropy flux pair $(H, Q)$, for any
  $k \in \reali$ and for any test function
  $\phi \in \Cc1 (\reali \times \reali^N; \reali_+)$
  \begin{equation}
    \label{eq:be}
    \begin{aligned}
      & \int_I \int_{\Omega} \left[ H \left(r (t,x), k
        \right) \partial_t \phi (t,x) + Q \left(t,x,r (t,x), k \right)
        \cdot \nabla \phi (t,x)\right] \d{x} \d{t}
      \\
      & - \int_I \int_{\Omega}
      \partial_1 H \left(r (t,x), k \right) \, r (t,x) \, \div u (t,x)
      \, \phi (t,x) \d{x} \d{t}
      \\
      & + \int_I \int_{\Omega} \div Q \left(t,x,r (t,x), k \right) \,
      \phi (t,x)\d{x} \d{t}
      \\
      & + \int_{\Omega} H\left(r_o (x), k \right) \, \phi (0,x) \d{x}
      % - \int_{\Omega}H\left(r (T,x), k \right) \, \phi (T,x)
      % \d{x}
      % \\
      % &
      + \norma{u}_{\L\infty (I \times \Omega; \reali^N)} \int_I
      \int_{\partial \Omega} H \left(0, k \right) \, \phi (t,\xi)
      \d\xi \d{t} \geq 0.
    \end{aligned}
  \end{equation}
\end{definition}

\begin{lemma}
  \label{lem:equiv}
  Let~$\boldsymbol{(\Omega)}$ and~\emph{\textbf{(u)}} hold. Assume
  $r_o \in (\L\infty \cap \BV) (\Omega;\reali)$. For
  $(t_o, x_o) \in I\times\Omega$, introduce the map
  \begin{equation}
    \label{eq:5}
    \begin{array}{rccl}
      X (\, \cdot \; ; t_o, x_o) \colon
      & I (t_o, x_o)
      & \to
      & \Omega
      \\
      & t
      & \to
      & X (t; t_o, x_o)
    \end{array}
    \quad \mbox{ solving } \quad
    \left\{
      \begin{array}{l}
        \dot x = u (t,x)
        \\
        x (t_o)  = x_o \,,
      \end{array}
    \right.
  \end{equation}
  $I (t_o,x_o)$ being the maximal interval where a solution to the
  Cauchy problem above is defined. The map $r$ defined by
  \begin{equation}
    \label{eq:6}
    r (t,x)
    =
    \left\{
      \begin{array}{l@{\qquad}r@{\,}c@{\,}l@{}}
        \displaystyle
        r_o\left(X (0; t,x)\right)
        \exp \left(
        - \int_0^t
        \div u\left(\tau, X (\tau; t,x)\right) \d\tau
        \right)
        & x
        & \in
        & X (t; 0, \Omega)
        \\
        0
        & x
        & \in
        & X (t; \left[0,t\right[, \partial\Omega)
      \end{array}
    \right.
  \end{equation}
  is a RE--solution to~\eqref{eq:2}. Moreover,
  $r \colon [0,T] \to (\L\infty \cap \BV) (\Omega;\reali)$ is
  $\L1$--continuous.
\end{lemma}

\begin{proof}
  We first regularise the initial datum,
  using~\cite[Theorem~1]{AnzellottiGiaquinta}, see
  also~\cite[Formula~(1.8)~and Theorem~1.17]{GiustiBook}: for
  $h \in \naturali \setminus \{0\}$, there exists a sequence
  $\tilde r_h \in \C\infty (\Omega;\reali)$ such that
  \begin{displaymath}
    \lim_{h \to +\infty}
    \norma{\tilde r_h- r_o}_{\L1 (\Omega;\reali)}=0,
    \qquad
    \norma{\tilde r_h}_{\L\infty (\Omega; \reali)} \leq
    \norma{r_o}_{\L\infty (\Omega; \reali)}
    \quad \mbox{and} \quad
    \lim_{h \to \infty}\tv (\tilde r_h) = \tv (r_o).
  \end{displaymath}
  Let $\Phi_h \in \Cc3 (\reali^N; [0,1])$ be such that
  $\Phi_h (\xi) = 1$ for all $\xi \in \partial\Omega$, $\Phi_h (x)=0$
  for all $x\in\Omega$ with $B (x, 1/h) \subseteq \Omega$ and
  $\norma{\nabla \Phi_h}_{\L\infty (\Omega;\reali^N)} \leq 2h$. Let
  \begin{equation}
    \label{eq:approxdi}
    r_o^h (x) = \left(1 - \Phi_h (x)\right) \, \tilde r_h (x)
    \quad \mbox{ for all } x\in \overline{\Omega},
  \end{equation}
  so that $r_o^h \in \C3 (\Omega;\reali)$.  By construction,
  $\lim_{h \to +\infty}\norma{r_o^h - r_o}_{\L1 (\Omega;\reali)} = 0$.
  Moreover, $r_o^h (\xi) = 0$ for all $\xi \in \partial\Omega$ and
  $h \in \naturali \setminus\{0\}$, and the following uniform bounds
  hold
  \begin{align}
    \label{eq:12}
    \norma{r_o^h}_{\L\infty (\Omega; \reali)} \leq \
    & \norma{r_o}_{\L\infty (\Omega; \reali)},
    \\
    \label{eq:13}
    \tv (r_o^h) \leq \
    & \mathcal{O} (1) \, \norma{r_o}_{\L\infty (\Omega; \reali)}
      +\tv (r_o).
  \end{align}
  Using the sequence $r_o^h$, define the corresponding
  sequence $r_h$ according to~\eqref{eq:6}. Obviously, each $r_h$ is a
  strong solution to~\eqref{eq:2}.  By~\cite[Proposition~6.2]{Elena1},
  each $r_h$ is also a RE--solution to~\eqref{eq:2}

  Let $r$ be defined as in~\eqref{eq:6}. It is clear that $r_h$
  converges to $r$ in $\L1$. Since Definition~\ref{def:besol} is
  stable under $\L1$ convergence, see~\cite{MalekEtAlBook, Elena1}, we
  obtain that $r$ is a RE--solution to~\eqref{eq:2}.

  The continuity in time of $r$ follows from the continuity in time of
  $r_h$ and the fact that $r$ is the uniform limit of $r_h$.
\end{proof}

The following Lemma extends to the case of the IBVP the results
in~\cite[Lemma~5.1, Corollary~5.2 and
Lemma~5.3]{ColomboHertyMercier}. Note that, due to the presence of the
boundary, this extension needs some care, see Remark~\ref{rem:note}.

\begin{lemma}
  \label{lem:properties}
  Let~$\boldsymbol{(\Omega)}$ and~\emph{\textbf{(u)}} hold. Assume
  $r_o \in (\L\infty \cap \BV) (\Omega;\reali)$. Then, the solution
  $r$ to~\eqref{eq:2} is such that
  $r \in \C{0,1} (I;\L1 (\Omega;\reali))$ and for all $t, \, s \in I$,
  \begin{align}
    \label{eq:rL1}
    \norma{r (t)}_{\L1 (\Omega;\reali)} \leq \
    & \norma{r_o}_{\L1 (\Omega;\reali)}
    \\
    \label{eq:rLinf}
    \norma{r (t)}_{\L\infty (\Omega;\reali)} \leq \
    & \norma{r_o}_{\L\infty (\Omega;\reali)} \,
      e^{\norma{\div \,u}_{\L1 ([0,t];\L\infty (\Omega;\reali))}}
    \\
    \label{eq:rTV}
    \tv\left(r (t)\right)
    \leq \
    & \exp\left(\int_0^t \norma{\nabla u \left(\tau\right)}_{\L\infty
      (\Omega;\reali^{N \times N})} \d\tau\right)
      \biggl( \mathcal{O} (1) \, \norma{r_o}_{\L\infty (\Omega;\reali)} \biggr.
    \\
    \nonumber
    &  \biggl. +
      \tv(r_o)+ \norma{r_o}_{\L1 (\Omega;\reali)} \,
      \int_0^t \norma{\nabla \div u (\tau)}_{\L\infty
      (\Omega;\reali^N)} \d\tau \biggr),
    \\
    \label{eq:lipT}
    \norma{r (t) - r (s)}_{\L1 (\Omega;\reali)} \leq \
    & \tv\left(r\left(\max\{t,s\}\right)\right) \, \modulo{t-s} \,.
  \end{align}
  If also $\tilde r_o \in (\L\infty \cap \BV) (\Omega;\reali)$ and
  $\tilde r$ is the corresponding solution to~\eqref{eq:2}, for all
  $t \in I$,
  \begin{equation}
    \label{eq:dipLip}
    \norma{r (t) - \tilde r (t)}_{\L1 (\Omega;\reali)}
    \leq
    \norma{r_o - \tilde r_o}_{\L1 (\Omega;\reali)} \,.
  \end{equation}
\end{lemma}

 \begin{proof}
   The proofs of~\eqref{eq:rL1} and~\eqref{eq:rLinf} directly follow
   from~\eqref{eq:6}. In particular, to get~\eqref{eq:rL1}, exploit
   the change of variable $y=X (0;t,x)$, so that $x=X (t;0,y)$,
   see~\cite[\S~5.1]{ColomboHertyMercier}. Note that if
   $x\in X (t;0,\Omega)$ then
   $y \in X\left(0; t, X (t; 0, \Omega) \right) \subseteq \Omega$.
   Denote the Jacobian of this change of variable by
   $J (t,y)=\det\left(\nabla_yX (t;0,y)\right)$. Then $J$ solves
   \begin{displaymath}
     \frac{\d{J (t,y)}}{\d{t}} = \div u \left(t,X (t;0,y)\right)  J (t,y)
     \quad \mbox{ with } \quad J (0,y)=1.
   \end{displaymath}
   Hence,
   $J (t,y) {=} \exp\left( \int_0^t \div u \left(\tau, X
       (\tau;0,y)\right) \d\tau \right)$, which implies $J (t,y)> 0$
   for $t \in [0,T]$ and $y \in \Omega$.

   To prove~\eqref{eq:rTV}, regularise the initial datum $r_o$ as in
   the proof of Lemma~\ref{lem:equiv}: $r_o^h \in \C3 (\Omega;\reali)$
   converges to $r_o$ in $\L1 (\Omega;\reali)$, $r_o^h (\xi)=0$ for
   all $\xi \in \partial \Omega$ and~\eqref{eq:12}--\eqref{eq:13}
   hold.
   % and $\tv (r_o^h) \leq \mathcal{O} (1) \norma{r_o}_{\L\infty
   % (\Omega;\reali)} + \tv (r_o)$.

   Using the sequence $r_o^h$, define according to~\eqref{eq:6} the
   corresponding sequence $r_h$ of solutions to~\eqref{eq:2}. Observe
   that $r_h (t) \in \C1 (\Omega; \reali)$ for every $t \in [0,T]$.
   Proceed similarly to the proof
   of~\cite[Lemma~5.4]{ColomboHertyMercier}: differentiate the
   solution to~\eqref{eq:5} with respect to the initial point, that is
   \begin{align*}
     \nabla_x X (\tau;t,x) = \
     &
       \boldsymbol{\Id} +
       \int_t^\tau \nabla_x u \left(t,X (s;t,x)\right) \nabla_x X (s;t,x) \d{s},
     \\
     \norma{\nabla_x X (\tau;t,x)} \leq \
     & 1 + \int_\tau^t \norma{\nabla_x u \left(t,X (s;t,x)\right)}
       \norma{\nabla_x X (s;t,x)} \d{s},
   \end{align*}
   since $\tau \in [0,t]$, so that, applying Gronwall Lemma,
   \begin{displaymath}
     \norma{\nabla_x X (\tau;t,x)} \leq
     \exp\left(
       \int_\tau^t \norma{\nabla_x u (s)}_{\L\infty (\Omega;\reali^{N \times N})}\d{s}
     \right).
   \end{displaymath}
   By~\eqref{eq:6} and the properties of $r_o^h$, the gradient of
   $r_h (t)$ is well defined (and continuous) on $\Omega$: in
   particular,
   \begin{align*}
     \nabla r_h (t,x) = \
     & \exp\left(\int_0^t - \div u \left(\tau,X (\tau;t,x)\right) \d\tau\right)
       \biggl(\nabla r_o^h \left(X (0;t,x)\right) \, \nabla_x X (0;t,x)
     \\
     & \qquad
       \left. -
       r_o^h\left(X (0;t,x)\right) \,
       \int_0^t \nabla \div u \left(\tau,X (\tau;t,x)\right) \,
       \nabla_x X (\tau;t,x) \d\tau
       \right).
   \end{align*}
   Hence, for every $t \in I$, using again the change of variable
   described at the beginning of the proof,
   \begin{equation}
     \label{eq:nablar}
     \begin{aligned}
       \!\!\norma{\nabla r_h (t)}_{\L1 (\Omega;\reali^N)} \leq \ &
       \exp\left(\int_0^t \norma{\nabla u \left(\tau\right)}_{\L\infty
           (\Omega;\reali^{N \times N})}\d\tau\right)
       \\
       & \times \left( \int_{\Omega} \modulo{\nabla r_o^h (x)}\d{x} +
         \norma{r_o^h}_{\L1 (\Omega;\reali)} \! \int_0^t \norma{\nabla
           \div u (\tau)}_{\L\infty (\Omega;\reali^N)} \d\tau \right).
     \end{aligned}
   \end{equation}

   Let $r$ be defined as in~\eqref{eq:6}: clearly, $r_h \to r$ in
   $\L1 (\Omega;\reali)$. Due to the lower semicontinuity of the total
   variation, to~\eqref{eq:nablar} and to the hypotheses on the
   approximation $r_o^h$, for $t \in I$ we get
   \begin{align*}
     \tv \left(r (t)\right) \leq \
     & \lim_h \tv \left(r_h (t) \right)
       = \lim_h \norma{\nabla r_h (t)}_{\L1 (\Omega; \reali^N)}
     \\
     \leq \
     & \exp\left(\int_0^t \norma{\nabla u \left(\tau\right)}_{\L\infty
       (\Omega;\reali^{N \times N})}\d\tau\right)
     \\
     & \times \left( \lim_h \tv(r_o^h)+ \lim_h \norma{r_o^h}_{\L1 (\Omega;\reali)} \,
       \int_0^t \norma{\nabla \div u (\tau)}_{\L\infty
       (\Omega;\reali^N)} \d\tau \right)
     \\
     \leq \
     & \exp\left(\int_0^t \norma{\nabla u \left(\tau\right)}_{\L\infty
       (\Omega;\reali^{N \times N})}\d\tau\right)
     \\
     & \times \left( \mathcal{O} (1)\norma{r_o}_{\L\infty (\Omega;\reali)}
       +
       \tv(r_o)+ \norma{r_o}_{\L1 (\Omega;\reali)} \,
       \int_0^t \norma{\nabla \div u (\tau)}_{\L\infty
       (\Omega;\reali^N)} \d\tau \right),
   \end{align*}
   concluding the proof of~\eqref{eq:rTV}. The proof of the
   $\L1$--Lipschitz continuity in time is done analogously, leading
   to~\eqref{eq:lipT}.

   Finally, \eqref{eq:dipLip} follows from~\eqref{eq:rL1}, due to the
   linearity of~\eqref{eq:2}.
 \end{proof}

 \begin{remark}
   \label{rem:note}
   {\rm We underline that the total variation estimate just obtained
     differs from that presented
     in~\cite[Lemma~5.3]{ColomboHertyMercier}, where the transport
     equation $\partial_t r + \div \left( r \, u (t,x)\right) = 0$ is
     studied not on a bounded domain $\Omega$, but on all
     $\reali^N$. Indeed, compare~\eqref{eq:rTV}
     and~\cite[Formula~(5.12)]{ColomboHertyMercier}: it is immediate
     to see that, in the case of a divergence free vector field $u$,
     the $\L\infty$--norm of the initial datum is still present in our
     case, while it is not
     in~\cite[Formula~(5.12)]{ColomboHertyMercier}. This is actually
     due to the presence of the boundary.

     Consider the following example to see the importance of the term
     $\norma{r_o}_{\L\infty (\Omega;\reali)}$ in~\eqref{eq:rTV}. Let
     $\Omega = B (0,1) \subset \reali^N$, $u (t,x)= - x$ and
     $r_o (x)=2$ for every $x \in \Omega$. Then, the solution
     to~\eqref{eq:5} is $X (t;t_o,x_o) = x_o \, e^{t_o-t}$. Since
     $\div u = -N$, the solution to~\eqref{eq:2} is:
     \begin{displaymath}
       r (t,x) =
       \begin{cases}
         2 \, e^{N \, t} & \mbox{ for } x \in B (0,e^{-t})
         \\
         0 & \mbox{elsewhere}.
       \end{cases}
     \end{displaymath}
     Therefore, for every $t \in \reali_+$, the total variation of
     $r (t)$ has contribution only from the \emph{jump} between
     $2 \, e^{N \, t}$ and $0$, multiplied by the $(N-1)$ dimensional
     measure of the boundary $\partial B (0,e^{-t})$, that is
     \begin{displaymath}
       % \label{eq:tv1}
       \tv \left(r (t)\right)
       = 2 \, e^{N \, t} \, \frac{2 \, \pi^{N/2} \, (e^{-t})^{N-1} }{\Gamma (N/2)}
       = 2  \, \frac{\pi^{N/2}}{\Gamma (N/2)} \, e^t,
     \end{displaymath}
     $\Gamma$ being the gamma function. Coherently,
     applying~\eqref{eq:rTV} we get
     \begin{displaymath}
       % \label{eq:tv2}
       \tv \left(r (t)\right) \leq
       e^t \, \mathcal{O} (1) \, \norma{r_o}_{\L\infty (\Omega;\reali)}
       = 2  \, \mathcal{O} (1) \, e^t \,,
     \end{displaymath}
     which confirms the necessity of the term
     $\norma{r_o}_{\L\infty (\Omega;\reali)}$ in the right hand side
     of~\eqref{eq:rTV}.}
 \end{remark}

We now provide a stability estimate of use below.

 \begin{lemma}
   \label{lem:stab}
   Let~$\boldsymbol{(\Omega)}$ hold. Let $u$ and $\tilde u$
   satisfy~\emph{\textbf{(u)}}. Assume
   $r_o \in (\L\infty \cap \BV) (\Omega;\reali)$. Call $r$ and
   $\tilde r$ the solutions to~\eqref{eq:2} obtained with $u$ and
   $\tilde u$, respectively. Then, for all $t \in I$,
   \begin{align}
     \nonumber
     &  \norma{r (t) - \tilde r (t)}_{\L1 (\Omega;\reali)}
     \\ \label{eq:14}
     \leq \
     & e^{\kappa (t)}
       \int_0^t
       \norma{(u - \tilde u) (s)}_{\L\infty (\Omega;\reali^N)}\d{s}
       \left[
       \mathcal{O} (1) \, \norma{r_o}_{\L\infty (\Omega;\reali)}
       +
       \tv (r_o)
       +
       \norma{r_o}_{\L1 (\Omega;\reali)} \,
       \kappa_1 (t)
       \right]
     \\ \nonumber
     &
       + \norma{r_o}_{\L1 (\Omega;\reali)}
       \int_0^t  \norma{\div (u-\tilde u) (s)}_{\L\infty (\Omega;\reali)}\d{s},
   \end{align}
   where
   \begin{align*}
     \kappa (t) = \
     & \int_0^t\max\left\{ \norma{\nabla u (s)}_{\L\infty
       (\Omega;\reali^{N\times N})}, \norma{\nabla \tilde u
       (s)}_{\L\infty (\Omega;\reali^{N\times N})} \right\} \d{s},
     \\
     \kappa_1 (t) = \
     & \int_0^t
       \max\left\{
       \norma{\nabla \div u (s)}_{\L\infty (\Omega;\reali^{N})},
       \norma{\nabla \div \tilde u (s)}_{\L\infty (\Omega;\reali^{N})}
       \right\} \d{s}.
   \end{align*}
 \end{lemma}
 \begin{proof}
   Regularise the initial datum $r_o$ as in the proof of
   Lemma~\ref{lem:equiv}: for any $h \in \naturali\setminus\{0\}$ we
   have that $r_o^h\in\C3 (\Omega;\reali)$ converges to $r_o$ in
   $\L1 (\Omega; \reali)$, $r_o^h (\xi) = 0$ for all
   $\xi \in \partial\Omega$ and~\eqref{eq:12}--\eqref{eq:13} hold.

   For $\theta \in [0,1]$, set
   \begin{displaymath}
     u_\theta (t,x) =  \theta \, u (t,x) + (1-\theta) \, \tilde u (t,x).
   \end{displaymath}
   Call $r_\theta^h$ the solution to~\eqref{eq:2} corresponding to the
   vector field $u_\theta$ above and to the initial datum
   $r_o^h$. Consider the map $X_\theta$ associated to $u_\theta$, as
   in~\eqref{eq:5}. We have that
   $r_\theta^h (t) \in \C1 (\Omega;\reali)$ for every $t \in I$ and it
   satisfies~\eqref{eq:6}, that now reads as follow:
   \begin{equation}
     \label{eq:10}
     \!\!\! r_\theta^h (t,x)  =
     \left\{\!\!\!
       \begin{array}{lr}
         \displaystyle
         r_o^h \left(X_\theta (0;t,x)\right)
         \exp\left[
         -
         \int_0^t
         \div u_\theta \left(\tau,X_\theta (\tau;t,x)\right)
         \d\tau
         \right]
         & \mbox{if }x \in X_\theta (t;0,\Omega)
         \\
         0
         & \mbox{elsewhere.}
       \end{array}
     \right.
   \end{equation}
   Derive the analog of~\eqref{eq:5} with respect to $\theta$ and
   recall that $X_\theta (t;t,x)=x$ for all $\theta$:
   \begin{displaymath}
     \left\{
       \begin{array}{l}
         \partial_t \partial_\theta X_\theta (\tau;t,x) =
         u (\tau,X_\theta(\tau;t,x)) - \tilde u (\tau,X_\theta(\tau;t,x))
         + \nabla u_\theta (\tau,X_\theta(\tau;t,x))
         \; \partial_\theta X_\theta (\tau;t,x)
         \\
         % X_\theta (t) = x \quad \Longrightarrow
         % \quad
         \partial_\theta X_\theta (t;t,x) = 0 \,.
       \end{array}
     \right.
   \end{displaymath}
   The solution to this problem is given by
   \begin{align}
     \nonumber
     \partial_\theta X_\theta (\tau;t,x)
     =
     & \int_t^\tau \exp\left(\int_s^\tau
       \nabla u_\theta (\sigma,X_\theta(\sigma;t,x)) \d\sigma\right)\!\!
       \left(u \left(s,X_\theta(s;t,x)\right) -
       \tilde u \left(s,X_\theta(s;t,x)\right) \right) \d{s}
     \\
     \label{eq:dthxth}
     =
     &
       \int_\tau^t \exp\left(\int_\tau^s
       -\nabla u_\theta (\sigma,X_\theta(\sigma;t,x)) \d\sigma\right)
       (\tilde u - u)\! \left(s,X_\theta(s;t,x)\right)\d{s}.
   \end{align}
   Derive now the non zero expression in the right hand side
   of~\eqref{eq:10} with respect to $\theta$:
   \begin{align*}
     & \partial_\theta r_\theta^h (t,x)
     \\
     = \
     &
       \exp\left(
       \int_0^t
       - \div u_\theta \left(\tau,X_\theta (\tau;t,x)\right)
       \d\tau
       \right)
     \\
     & \times
       \biggl\{
       \nabla r_o^h \left(X_\theta (0;t,x)\right)
       \partial_\theta X_\theta (0;t,x)
       + r_o^h\left(X_\theta (0;t,x)\right)
       \int_0^t
       \div(\tilde u-u)\!
       \left(\tau, X_\theta (\tau;t,x)\right)
       \d\tau
     \\
     & \qquad
       \left.
       - r_o^h\left(X_\theta (0;t,x)\right)
       \int_0^t
       \nabla \div u_\theta
       \left(\tau,X_\theta (\tau;t,x)\right)
       \cdot
       \partial_\theta X_\theta (\tau;t,x)
       \d\tau
       \right\}
     \\
     = \
     & \exp\left(
       \int_0^t
       -\div u_\theta \left(\tau,X_\theta (\tau;t,x)\right)
       \d\tau
       \right)
     \\
     & \times
       \biggl\{
       \nabla r_o^h\left(X_\theta (0;t,x)\right)
       \int_0^t \exp \left(\int_0^s
       -\nabla u_\theta (\sigma,X_\theta(\sigma;t,x)) \d\sigma\right)
       (\tilde u -u)\! \left(s,X_\theta(s;t,x)\right) \d{s}
     \\
     &
       + r_o^h \left(X_\theta (0;t,x)\right)
       \int_0^t
       \div(\tilde u-u)\!
       \left(\tau, X_\theta (\tau;t,x)\right)
       \d\tau
     \\
     &
       - r_o^h \left(X_\theta (0;t,x)\right)
       \int_0^t
       \nabla \div u_\theta \left(\tau,X_\theta (\tau;t,x)\right)
     \\
     & \quad \times
       \left[ \int_\tau^t \exp\left(\int_\tau^s
       -\nabla u_\theta (\sigma,X_\theta(\sigma;t,x)) \d\sigma\right)
       (\tilde u - u) \! \left(s,X_\theta(s;t,x)\right) \d{s}\right]
       \d\tau
       \biggr\}\,,
   \end{align*}
   where we used~\eqref{eq:dthxth}. Call $r^h$ and
   $\tilde r^h$ the solutions to~\eqref{eq:2} corresponding to
   velocities $u$ and $\tilde u$ respectively, and initial datum
   $r_o^h$: in other words, $r^h = r^h_{\theta=1}$, while
   $\tilde r^h = r^h_{\theta=0}$. Compute
   \begin{equation}
     \label{eq:11}
     \norma{r^h (t) - \tilde r^h (t)}_{\L1 (\Omega;\reali)}
     \leq
     \int_{\Omega}
     \modulo{\int_0^1 \partial_\theta r_\theta^h (t,x) \d\theta} \d{x}
     \leq
     \int_0^1
     \int_{X_\theta (t;0,\Omega)} \modulo{\partial_\theta r_\theta^h (t,x)}\d{x}\d\theta.
   \end{equation}
   In particular, introduce the change of variable for $X_\theta$
   analogous to that presented at the beginning of the proof of
   Lemma~\ref{lem:properties},
   % introduce the change of variable
   % $y=X_\theta (0;t,x)$, so that $x=X_\theta (t;0,y)$. Obviously, if
   % $x\in X_\theta (t;0,\Omega)$ then $y\in\Omega$. Denote the
   % Jacobian
   % of this change of variable by
   % $J (t,y)=\det\left(\nabla_yX_\theta (t;0,y)\right)$. Then $J$
   % solves
   % \begin{displaymath}
   %   \frac{\d{J (t,y)}}{\d{t}} = \div u_\theta \left(t,X_\theta
   %     (t;0,y)\right) J (t,y)
   %   \quad \mbox{ with } \quad J (0,y)=1.
   % \end{displaymath}
   % Hence
   % $J (t,y) = \exp\left( \int_0^t \div u_\theta \left(\tau,X_\theta
   %     (\tau;0,y)\right) \d\tau \right)$,
   % which implies $J (t,y)> 0$ for all $t \in [0,T]$ and $y \in
   % \Omega$. Exploiting this change of variable,
   set $Y = X_\theta \left( 0; t, X_\theta (t; 0, \Omega)\right)$ and
   compute
   \begin{align*}
     & \int_{X_\theta (t;0,\Omega)} \modulo{\partial_\theta r_\theta^h (t,x)}\d{x}
     % \\
     % = \
     % &
     % \int_{\Omega}
     % \modulo{\partial_\theta r_\theta^h \left(t,X_\theta
     %   (t;0,y)\right)}
     % J (t,y)
     % \d{y}
     \\
     \leq \
     &
       \int_Y
       \modulo{\nabla r_o^h(y)
       \int_0^t \exp\left(\int_0^s
       -\nabla u_\theta \left(\sigma,X_\theta (\sigma;0,y)\right) \d\sigma\right)
       \,
       (\tilde u -u)\!\left(s,X_\theta (s;0,y)\right) \d{s}}\d{y}
     \\
     &
       + \int_Y
       \modulo{r_o^h (y)
       \int_0^t
       \div(\tilde u -u)\!\left(\tau,X_\theta (\tau;0,y)\right)\d\tau}\d{y}
     \\
     &
       + \int_Y
       \left|r_o^h (y) \,
       \int_0^t \nabla\div u_\theta\left(\tau,X_\theta (\tau;0,y)\right)
       \right.
     \\
     &
       \qquad\quad
       \left. \times
       \int_\tau^t
       \exp\left(
       \int_\tau^s
       -\nabla u_\theta \left(\sigma,X_\theta (\sigma;0,y)\right) \d\sigma
       \right)
       (\tilde u -u)\!\left(s,X_\theta (s;0,y)\right)\d{s}\d\tau\right|
       \d{y}
     \\
     \leq \
     &
       \left(\int_\Omega \modulo{\nabla r_o^h(y)}\d{y}\right)
       \,
       \exp\left(\int_0^t \norma{\nabla u_\theta (s)}_{\L\infty (\Omega;\reali^{N\times N})}
       \d{s} \right)
       \,
       \int_0^t \norma{(u-\tilde u) (s)}_{\L\infty (\Omega;\reali^N)} \d{s}
     \\
     &
       +
       \norma{r_o^h}_{\L1 (\Omega;\reali)} \,
       \int_0^t
       \norma{\div (u-\tilde u) (s)}_{\L\infty (\Omega;\reali)}\d{s}
     \\
     &
       +
       \norma{r_o^h}_{\L\infty (\Omega;\reali)}
       \,
       \int_0^t \norma{\nabla\div u_\theta (s)}_{\L1 (\Omega;\reali^N)} \d{s}
     \\
     &
       \quad
       \times
       \exp\left(\int_0^t \norma{\nabla u_\theta (s)}_{\L\infty (\Omega;\reali^{N\times N})}
       \d{s} \right)
       \,
       \int_0^t \norma{(u-\tilde u) (s)}_{\L\infty (\Omega;\reali^N)} \d{s}.
   \end{align*}
   Therefore, inserting the latter result above in~\eqref{eq:11}
   yields
   \begin{align}
     \label{eq:14a}
     &  \norma{r^h (t) - \tilde r^h (t)}_{\L1 (\Omega;\reali)}
     \\
     \nonumber
     \leq
     & \exp \! \left(\int_0^t\max\left\{
       \norma{\nabla u (s)}_{\L\infty (\Omega;\reali^{N\times N})},
       \norma{\nabla \tilde u (s)}_{\L\infty (\Omega;\reali^{N\times N})}
       \right\} \d{s}\right)
       \! \int_0^t
       \norma{(u - \tilde u) (s)}_{\L\infty (\Omega;\reali^N)}\d{s}
     \\
     \nonumber
     & {\times}
       \left[
       \int_\Omega \modulo{\nabla r_o^h(y)}\d{y}
       +
       \norma{r_o^h}_{\L\infty (\Omega;\reali)}
       \int_0^t
       \max\left\{
       \norma{\nabla \div u (s)}_{\L1 (\Omega;\reali^{N})},
       \norma{\nabla \div \tilde u (s)}_{\L1 (\Omega;\reali^{N})}
       \right\} \d{s}
       \right]
     \\
     \label{eq:14b}
     &
       + \norma{r_o^h}_{\L1 (\Omega;\reali)}
       \int_0^t  \norma{\div (u-\tilde u) (s)}_{\L\infty (\Omega;\reali)}\d{s}.
   \end{align}
   We now let $h$ tend to $+\infty$. We know that $r_o^h$ converges to
   $r_o$ in $\L1 (\Omega;\reali)$, so that $r_\theta^h$, solution
   to~\eqref{eq:2} with velocity $u_\theta$ and initial datum $r_o^h$,
   converges to a function $r_\theta$ in $\L1$ which is solution
   to~\eqref{eq:2} with velocity $u_\theta$ and initial datum $r_o$.
   Call $r = r_{\theta=1}$ and $\tilde r = r_{\theta=0}$: they are
   solutions to~\eqref{eq:2} with velocities $u$ and $\tilde u$
   respectively, and initial datum $r_o$. It is clear that $r^h \to r$
   and $\tilde r^h \to \tilde r$ in $\L1$. Therefore, the
   inequality~\eqref{eq:14a}--\eqref{eq:14b} in the limit
   $h \to +\infty$ reads
   \begin{align*}
     &  \norma{r (t) - \tilde r (t)}_{\L1 (\Omega;\reali)}
     \\
     \leq
     &
       \exp\left(\int_0^t\max\left\{
       \norma{\nabla u (s)}_{\L\infty (\Omega;\reali^{N\times N})},
       \norma{\nabla \tilde u (s)}_{\L\infty (\Omega;\reali^{N\times N})}
       \right\} \d{s}\right)
       \int_0^t
       \norma{(u - \tilde u) (s)}_{\L\infty (\Omega;\reali^N)}\d{s}
     \\
     & \times
       \biggl[
       \mathcal{O} (1) \, \norma{r_o}_{\L\infty (\Omega;\reali)}
       +
       \tv (r_o)
     \\
     & \qquad \left.
       +
       \norma{r_o}_{\L\infty (\Omega;\reali)}
       \int_0^t
       \max\left\{
       \norma{\nabla \div u (s)}_{\L1 (\Omega;\reali^{N})},
       \norma{\nabla \div \tilde u (s)}_{\L1 (\Omega;\reali^{N})}
       \right\} \d{s}
       \right]
     \\
     &
       + \norma{r_o}_{\L1 (\Omega;\reali)}
       \int_0^t  \norma{\div (u-\tilde u) (s)}_{\L\infty (\Omega;\reali)}\d{s},
   \end{align*}
   where we used the fact that
   $\displaystyle\int_\Omega \modulo{\nabla r_o^h (y)}\d{y} = \tv
   (r_o^h)$ and~\eqref{eq:13}.
 \end{proof}

 \begin{proofof}{Theorem~\ref{thm:fixpt}}
   The proof relies on a fixed point argument and consists of several
   steps. Fix
   $R = \max\left\{\norma{\rho_o}_{\L1 (\Omega;\reali^n)}, \,
     \norma{\rho_o}_{\L\infty (\Omega;\reali^n)}, \,
     \tv\left(\rho_o\right) \right\}$.
   Given a map $\mathcal{F} (t) \in \C0 (I;\reali_+)$, whose precise
   choice is given in the sequel, the following functional space is of
   use below:
   \begin{equation}
     \label{eq:X}
     \mathcal{X}_R
     =
     \left\{
       r \in \C0(I; \L1 (\Omega;\reali^n)) \colon
       \begin{array}{l}
         \norma{r}_{\L\infty (I;\L1 (\Omega;\reali^n))} \leq R
         \mbox{ and }
         \\
         \norma{r (t)}_{\L\infty (\Omega;\reali^n)} < +\infty
         \mbox{ for all } t \in I
         \\
         \tv \left(r (t)\right) \leq \mathcal{F} (t)
         \mbox{ for all } t \in I
       \end{array}
     \right\}
   \end{equation}
   with the distance
   $d (\rho_1, \rho_2) = \norma{\rho_1 - \rho_2}_{\L\infty (I;\L1
     (\Omega;\reali^n))}$,
   so that $\mathcal{X}_R$ is a complete metric space.

   Throughout, we denote by $C$ a positive constant that depends on
   the assumptions~\textbf{($\boldsymbol{\Omega}$)}, \textbf{(V)},
   \textbf{(J)}, on $R$ and on $n$. The constant $C$ does not depend on
   time. For the sake of simplicity, introduce the notation
   $\Sigma_t = [0,t] \times \Omega \times \reali^m$.

  \paragraph{Reduction to a Fixed Point Problem.}
  Define the map
  \begin{equation}
    \label{eq:7}
    \begin{array}{ccccl}
      \mathcal{T}
      & \colon
      & \mathcal{X}_R
      & \to
      & \mathcal{X}_R
      \\
      &
      & r
      & \to
      & \rho
    \end{array}
  \end{equation}
  where $\rho \equiv (\rho^1, \ldots, \rho^n)$ solves
  \begin{equation}
    \label{eq:3}
    \left\{
      \begin{array}{l@{\qquad}r@{\,}c@{\,}l}
        \partial_t \rho^i
        +
        \div \left[
        \rho^i \; V^i \left(t,x,\left(\mathcal{J}^i r (t)\right) (x)\right)
        \right]
        =
        0
        & (t,x)
        & \in
        & I \times \Omega
          \qquad i = 1, \ldots, n
        \\
        \rho (t, \xi) = 0
        & (t, \xi)
        & \in
        & I \times \partial \Omega
        \\
        \rho (0,x) = \rho_o (x)
        & x
        & \in
        & \Omega.
      \end{array}
    \right.
  \end{equation}
  A map $\rho \in \mathcal{X}_R$ solves~\eqref{eq:1} in the sense of
  Definition~\ref{def:sol} if and only if $\rho$ is a fixed point for
  $\mathcal{T}$.

  \paragraph{$\mathcal{T}$ is Well Defined.}
  Given $r \in \mathcal{X}_R$, by~\textbf{(V)} and~\textbf{(J)}, for
  $i=1, \ldots, n$ each map
  \begin{equation}
    \label{eq:15}
    u^i (t,x)
    =
    V^i\left(t,x,\left(\mathcal{J}^i r (t)\right) (x)\right)
  \end{equation}
  satisfies~\textbf{(u)}. The solution $\rho$ to~\eqref{eq:3} is well
  defined, unique and belongs to $\C0 (I; \L1 (\Omega;\reali^n))$.
  With the notation introduced above, by~\eqref{eq:rL1} in
  Lemma~\ref{lem:properties}, for all $t \in I$,
  \begin{equation}
    \label{eq:8}
    \norma{\rho (t)}_{\L1 (\Omega;\reali^n)}
    \leq
    \norma{\rho_o}_{\L1 (\Omega;\reali^n)}
  \end{equation}
  and, by~\textbf{(V)}, \textbf{(J)} and~\eqref{eq:rLinf},
  \begin{align}
    \nonumber
    \norma{\rho^i (t)}_{\L\infty (\Omega;\reali)}
    \leq
    & \norma{\rho^i_o}_{\L\infty (\Omega;\reali)}
      \! \exp \! \left[
      t \norma{\div V^i}_{\L\infty (\Sigma_{t}; \reali)}
      {+}
      t K \norma{\nabla_w V^i}_{\L\infty (\Sigma_{t};  \reali^{N\times m})}
      \norma{r (t)}_{\L1 (\Omega;\reali^n)}
      \right]
    \\
    \nonumber
    \leq
    & \norma{\rho_o^i}_{\L\infty (\Omega;\reali)} \exp\left(
      t \, \mathcal{V} \left(
      1 +  K \, R
      \right)
      \right)
    \\
    \nonumber
    \leq \
    & \norma{\rho_o^i}_{\L\infty (\Omega;\reali)} \, e^{C\, t}
      \qquad\qquad \mbox{ for } i=1, \ldots,n, \mbox{ so that}
    \\
    \label{eq:18}
    \norma{\rho (t)}_{\L\infty (\Omega;\reali^n)}
    \leq \
    & \norma{\rho_o}_{\L\infty (\Omega;\reali^n)} \, e^{C\, t} \,.
  \end{align}
  Applying~\eqref{eq:rTV} in Lemma~\ref{lem:properties}, with the help
  of~\textbf{(V)} and~\textbf{(J)}, for all $t \in I$ and all
  $i=1,\ldots, n$,
  \begin{align}
    \label{eq:tvn1}
    \tv \left(\rho^i (t)\right)
    % \nonumber
    \leq
    & \exp\left(\int_0^t \norma{\nabla u^i \left(\tau\right)}_{\L\infty
      (\Omega;\reali^{N \times N})}\d\tau\right)
    \\ \nonumber
    & \times \left( \mathcal{O} (1)\norma{\rho_o^i}_{\L\infty (\Omega;\reali)}
      +
      \tv(\rho_o^i)+ \norma{\rho^i_o}_{\L1 (\Omega;\reali)} \,
      \int_0^t
      \norma{\nabla \div u^i (\tau)}_{\L\infty (\Omega;\reali^N)} \d\tau \right)
    \\ \nonumber
    \leq
    & \exp\left(t \norma{\nabla V^i}_{\L\infty(\Sigma_t;\reali^{N \times N})}
      + t \, K \, \norma{\nabla_w V^i}_{\L\infty (\Sigma_t; \reali^{N\times m})}
      \norma{r (t)}_{\L1 (\Omega;\reali^n)} \right)
    \\ \nonumber
    & \times \biggl[ \mathcal{O} (1)\norma{\rho^i_o}_{\L\infty (\Omega;\reali)}
      +
      \tv(\rho^i_o) + t \, \norma{\rho^i_o}_{\L1 (\Omega;\reali)}
      \biggl(
      \norma{\nabla_x \div V^i}_{\L\infty (\Sigma_t; \reali^N)}
    \\
    \nonumber
    & \qquad +
      K \, \left(\norma{\nabla_w \div V^i}_{\L\infty (\Sigma_t; \reali^m)}
      +
      \norma{\nabla_x \nabla_w V^i}_{\L\infty (\Sigma_t; \reali^{N\times m \times N})}\right)
      \norma{r (t)}_{\L1 (\Omega;\reali^n)}\biggr.
    \\
    \nonumber
    & \qquad
      + K^2 \,
      \norma{\nabla^2_{ww} V^i}_{\L\infty (\Sigma_t;\reali^{N \times m \times m})} \,
      \norma{r (t)}^2_{\L1 (\Omega;\reali^n)}
    \\
    \nonumber
    & \qquad
      \left.\biggl.
      +
      \norma{\nabla_w V^i}_{\L\infty (\Sigma_t; \reali^{N \times m})} \,
      \mathcal{K}\! \left(\norma{r (t)}_{\L1 (\Omega;\reali^n)}\right)\,
      \norma{r (t)}_{\L1 (\Omega;\reali^n)}
      \biggr)
      \right]
    \\
    \label{eq:tvn2}
    \leq
    & \left(C \, t + C \, \norma{\rho^i_o}_{\L\infty (\Omega;\reali)}
      + \tv (\rho_o^i)\right)
      \, e^{C t} \,,
  \end{align}
  so that
  \begin{equation}
    \label{eq:tvn3}
    \tv \left(\rho (t)\right)
    \leq
    \left(C \, t + C \, \norma{\rho_o}_{\L\infty (\Omega;\reali^n)}
      + \tv (\rho_o)\right)
    \, e^{C t} \,.
  \end{equation}
  The map $\mathcal{T}$ is thus well defined, setting in~\eqref{eq:X}
  \begin{equation}
    \label{eq:9}
    \mathcal{F} (t)
    =
    \left(C \, t + C \, \norma{\rho_o}_{\L\infty (\Omega;\reali^n)}
      + \tv (\rho_o)\right)
    \, e^{C t} \,.
  \end{equation}

  \paragraph{$\mathcal{T}$ is a Contraction.}
  For any $r_1, r_2 \in \mathcal{X}_R$, denote for $j = 1,2$,
  $\rho_j = \mathcal{T} (r_j)$ and, correspondingly, $u^i_j$ as
  in~\eqref{eq:15} for $i=1, \ldots, n$.  Compute, thanks
  to~\textbf{(V)} and~\textbf{(J)},
  \begin{align*}
    \norma{\nabla u^i_j (t)}_{\L\infty (\Omega;\reali^{N\times N})}
    \leq \
    & \norma{\nabla_x V^i}_{\L\infty (\Sigma_t;\reali^{N\times N})}
      +
      \norma{\nabla_w V^i}_{\L\infty (\Sigma_t;\reali^{N\times m})}
      \norma{\nabla_x \mathcal{J}^i r_j (t)}_{\L\infty (\Omega;\reali^{m\times N})}
    \\
    \leq \
    & \mathcal{V} \left(
      1
      +
      K \, \norma{r_j (t)}_{\L1 (\Omega;\reali^n)}
      \right)
    \\
    \leq \
    & \mathcal{V} \, (1 +  K \, R)
    \\
    \leq \
    & C
  \end{align*}
  and
  \begin{align*}
    & \norma{\nabla \div u^i_j (t)}_{\L\infty (\Omega;\reali^{N})}
    \\ \leq \
    & \norma{\nabla_x \div V^i}_{\L\infty (\Sigma_t; \reali^N)}
    \\
    & +
      \left(\norma{\nabla_w \div V^i}_{\L\infty (\Sigma_t; \reali^m)}
      +
      \norma{\nabla_x \nabla_w V^i}_{\L\infty (\Sigma_t; \reali^{N\times m \times N})}\right)
      \norma{\nabla_x \mathcal{J}^i r_j (t)}_{\L\infty (\Omega;\reali^{m\times N})}
    \\
    & + \norma{\nabla^2_w V^i}_{\L\infty (\Sigma_t;\reali^{N \times m \times m})} \,
      \norma{\nabla_x \mathcal{J}^i r_j (t)}^2_{\L\infty (\Omega;\reali^{m\times N})}
    \\
    & +
      \norma{\nabla_w V^i}_{\L\infty (\Sigma_t; \reali^{N \times m})} \,
      \norma{\nabla_x^2 \mathcal{J}^i r_j (t)}_{\L\infty (\Omega;\reali^{m \times N \times N})}
    \\
    \leq \
    & \mathcal{V}
      \left(
      1
      +
      K \, \norma{r_j (t)}_{\L1 (\Omega;\reali^n)}
      +
      K^2 \, \norma{r_j (t)}^2_{\L1 (\Omega;\reali^n)}
      +
      \mathcal{K}\! \left(\norma{r_j (t)}_{\L1 (\Omega;\reali^n)}\right)\,
      \norma{r_j (t)}_{\L1 (\Omega;\reali^n)}
      \right)
    \\
    \leq \
    & \mathcal{V}
      \left(
      1
      +
      K \, R
      +
      K^2 \, R^2
      +
      \mathcal{K}\! \left(R\right)\,
      R
      \right)
    \\
    \leq \
    & C \,.
  \end{align*}
  Furthermore, still using assumption~{\bf{(J)}}, we have that, for
  all $t \in I$,
  \begin{align*}
    \norma{(u^i_2 - u^i_1) (t)}_{\L\infty (\Omega;\reali^N)} \leq \
    &
      \norma{\nabla_w V^i}_{\L\infty (\Sigma_t; \reali^{N \times m})} \,
      \norma{\mathcal{J}^i r_2 (t)- \mathcal{J}^i r_1 (t)}_{\L\infty (\Omega;\reali^m)}
    \\
    \leq \
    & \mathcal{V} \,
      K \, \norma{r_2 (t) - r_1 (t)}_{\L1 (\Omega;\reali^n)}
    \\
    \leq \
    & C \, \norma{r_2 (t) - r_1 (t)}_{\L1 (\Omega;\reali^n)} \,.
    % \end{align*}
    % and
    % \begin{align*}
    \\
    \norma{\div(u^i_2 - u^i_1) (t)}_{\L\infty (\Omega;\reali)} \leq \
    & \norma{\nabla_w \div V^i}_{\L\infty (\Sigma_t;\reali^m)} \,
      \norma{\mathcal{J}^i r_2 (t) -
      \mathcal{J}^i r_1 (t)}_{\L\infty (\Omega;\reali^m)}
    \\
    & +
      \norma{\nabla_w V^i}_{\L\infty (\Sigma_t;\reali^{N \times m})} \,
      \norma{\nabla_x \mathcal{J}^i r_2 (t) -
      \nabla_x \mathcal{J}^i r_1 (t)}_{\L\infty (\Omega;\reali^{m\times N})}
    \\
    \leq \
    & \mathcal{V} \left(
      K
      +
      \mathcal{K} \left(\norma{r_1 (t)}_{\L1 (\Omega;\reali^n)}\right)
      \right) \norma{r_2 (t) - r_1 (t)}_{\L1 (\Omega;\reali^n)}
    \\
    \leq \
    & C \,
      \norma{r_2 (t) - r_1 (t)}_{\L1 (\Omega;\reali^n)} \,.
  \end{align*}
  Therefore, for all $t \in I$, by Lemma~\ref{lem:stab}, with obvious
  notation we have
  \begin{align*}
    &  \norma{\rho^i_2 (t) - \rho^i_1 (t)}_{\L1 (\Omega;\reali)}
    \\
    \leq \
    & e^{\kappa (t)}
      \int_0^t
      \norma{(u^i_2 - u^i_1) (s)}_{\L\infty (\Omega;\reali^N)}\d{s}
      \left[
      \mathcal{O} (1) \, \norma{\rho^i_o}_{\L\infty (\Omega;\reali)}
      +
      \tv (\rho^i_o)
      +
      \kappa_1 (t) \, \norma{\rho^i_o}_{\L1 (\Omega;\reali)}
      \right]
    \\
    & + \norma{\rho^i_o}_{\L1 (\Omega;\reali)}
      \int_0^t  \norma{\div (u^i_2 - u^i_1) (s)}_{\L\infty (\Omega;\reali)}\d{s}
    \\
    \leq \
    & C \, t
      \left[
      e^{C \, t}
      \left(
      \mathcal{O} (1) \, \norma{\rho_o}_{\L\infty (\Omega;\reali^n)}
      +
      \tv (\rho_o)
      +
      C \, t
      \right)
      + C
      \right]
      \norma{r_2 - r_1}_{\L\infty ([0,t]; \L1 (\Omega;\reali^n))}
    \\
    \leq \
    &   C \, t
      \left[
      e^{C \, t}
      \left(
      C \, \norma{\rho_o}_{\L\infty (\Omega;\reali^n)}
      +
      \tv (\rho_o)
      +
      C \, t
      \right)
      + C
      \right]
      \norma{r_2 - r_1}_{\L\infty ([0,t]; \L1 (\Omega;\reali^n))}
      \,,
  \end{align*}
  We obtain that $\mathcal{T}$ is a contraction when restricted to the
  time interval $[0, T_1]$, with $T_1$ such that
  \begin{equation}
    \label{eq:16}
    C \, T_1
    \left[
      e^{C \, T_1}
      \left(
        C \, \norma{\rho_o}_{\L\infty (\Omega;\reali^n)}
        +
        \tv (\rho_o)
        +
        C \, T_1
      \right)
      + C
    \right]
    = \frac12 \,.
  \end{equation}

  \paragraph{Existence of a solution on $[0, T_1]$.} By the steps
  above, there exists a fixed point $\rho_1 \in \mathcal{X}_R$ for the
  map $\mathcal{T}$ defined in~\eqref{eq:7}, restricted to functions
  defined on the time interval $[0, T_1]$. By construction, $\rho_1$
  solves~\eqref{eq:1} on the time interval $[0, T_1]$.

  \paragraph{Existence of a solution on $I$.} We consider two cases:
  $I=\reali_+$ and $I=[0,T]$, for a fixed positive $T$. If, in the
  second case, $T_1 \geq \sup I$, the statement obviously
  holds. Otherwise, if $T_1 < \sup I$, we extend $\rho_1$ to $I$ by
  iterating the procedure above.

  Assume that the solution exists up to the time $T_{k-1} < \sup I$.
  Thanks to the bounds~\eqref{eq:18} and~\eqref{eq:tvn2}, define
  recursively $T_k$ so that
  \begin{equation}
    \label{eq:4}
    \begin{aligned}
      C \, (T_k -T_{k-1}) \left[ \left(2 \, C \,
          \norma{\rho_o}_{\L\infty (\Omega;\reali^n)} + \tv (\rho_o) +
          C \, T_{k-1} \right) e^{C \, T_k }\right.  &
      \\
      \left.+ C \left( T_k - T_{k-1}\right) e^{C \, (T_k - T_{k-1})} +
        C \right] & = \frac12.
    \end{aligned}
  \end{equation}
  Indeed, the above procedure ensures that there exists a fixed point
  for the map $\mathcal{T}$ defined in~\eqref{eq:7}, restricted to
  functions defined on the time interval $[T_{k-1}, T_k]$. If, in the
  case of the time interval $I = [0,T]$, $T_k \geq \sup I$, the
  statement is proved. Otherwise, if we assume that the sequence
  $(T_k)$ remains less than $\sup I$, it is in particular
  bounded. Hence, the left hand side of the relation above tends to
  $0$, while the right hand side is $1/2 > 0$.  Therefore, the
  sequence $(T_k)$ is unbounded, ensuring that, for $k$ large, $T_k$
  is greater than $\sup I$, thus the solution to~\eqref{eq:1} is
  defined on all $I$.

  \paragraph{Bounds on the solution.} The $\L1$--bound follows
  immediately by the construction of the solution. By~\eqref{eq:18} we
  have
  \begin{align*}
    \norma{\rho^i (t)}_{\L\infty (\Omega;\reali)}
    \leq \
    &
      \norma{\rho^i_o}_{\L\infty (\Omega;\reali)}
      \exp\left(
      t \, \mathcal{V}
      \left(
      1 + K \,  \norma{\rho (t)}_{\L1 (\Omega;\reali^n)}
      \right)
      \right)
      \qquad \mbox{ whence}
    \\
    \norma{\rho (t)}_{\L\infty (\Omega;\reali^n)}
    \leq \
    &
      \norma{\rho_o}_{\L\infty (\Omega;\reali^n)}
      \exp\left(
      t \, \mathcal{V}
      \left(
      1 + K \,  \norma{\rho_o}_{\L1 (\Omega;\reali^n)}
      \right)
      \right).
  \end{align*}
  Moreover, by~\eqref{eq:tvn1}--\eqref{eq:tvn2}
  \begin{align*}
    \tv \! \bigl(\rho^i (t)\bigr)
    \leq
    & \exp\left(
      t \, \mathcal{V}
      \left(
      1 + K \,  \norma{\rho (t)}_{\L1 (\Omega;\reali^n)}
      \right)
      \right)
    \\
    & \times \!
      \biggl(
      \mathcal{O} (1) \norma{\rho^i_o}_{\L\infty (\Omega;\reali)}
      +
      \tv\left( \rho^i_o \right)
      +
      t \, \norma{\rho^i_o}_{\L1 (\Omega;\reali)} \,
      \mathcal{V}
    \\
    & \left. \times \!
      \left(
      1 {+} K \norma{\rho (t)}_{\L1 (\Omega;\reali^n)}
      + K^2  \norma{\rho (t)}_{\L1 (\Omega;\reali^n)}^2
      + \mathcal{K}\!\left( \norma{\rho (t)}_{\L1 (\Omega;\reali^n)}  \right)
      \norma{\rho (t)}_{\L1 (\Omega;\reali^n)}
      \right) \!
      \right) .
      % \end{align*}
      % so that
      % \begin{align*}
    \\
    \tv \bigl(\rho (t)\bigr)
    \leq
    & \ \exp\left(
      t \, \mathcal{V}
      \left(
      1 + K \,  \norma{\rho_o}_{\L1 (\Omega;\reali^n)}
      \right)
      \right)
    \\
    & \times
      \biggl(
      \mathcal{O} (1) n \, \norma{\rho_o}_{\L\infty (\Omega;\reali^n)}
      +
      \tv\left( \rho_o \right)
      +
      n \, t \, \norma{\rho_o}_{\L1 (\Omega;\reali^n)} \,
      \mathcal{V}
    \\
    & \left. \times
      \left(
      1 + K \,  \norma{\rho_o}_{\L1 (\Omega;\reali^n)}
      + K^2  \norma{\rho_o}_{\L1 (\Omega;\reali^n)}^2
      + \mathcal{K}\!\left( \norma{\rho_o}_{\L1 (\Omega;\reali^n)}  \right) \,
      \norma{\rho_o}_{\L1 (\Omega;\reali^n)}
      \right) \!
      \right),
  \end{align*}
  concluding the proof of~(\ref{item:bounds}).

  \paragraph{Lipschitz dependence on time.} Apply~\eqref{eq:lipT} in
  Lemma~\ref{lem:properties} and the total variation estimate obtained
  in the previous step: for any $t, \, s \in I$
  \begin{displaymath}
    \norma{\rho (t) - \rho (s)}_{\L1 (\Omega;\reali^n)}
    \leq
    \tv\left(\rho \left(\max\{t,s\}\right)\right) \, \modulo{t-s}.
  \end{displaymath}

  \paragraph{Lipschitz dependence on the initial datum.}
  Assume that $I = [0, t]$, so that
  $\lim\limits_{k\to +\infty} T_k = t$, where $T_k$ is defined
  recursively through~\eqref{eq:4}, which can be rewritten as follows:
  \begin{equation}
    \label{eq:19}
    C \, (T_k -T_{k-1}) \left[ \left( (2 \, C +1) \, R + C \, T_{k-1}
      \right) e^{C \, T_k } + C \left( T_k - T_{k-1}\right)
      e^{C \, (T_k - T_{k-1})} + C \right] = \frac12,
  \end{equation}
  the constant $C$ depending on the
  assumptions~\textbf{($\boldsymbol{\Omega}$)}, \textbf{(V)},
  \textbf{(J)} and on $R$, which is now defined as
  \begin{displaymath}
    R = \max\left\{
      \norma{\rho_o}_{\L1 (\Omega;\reali^n)}, \,
      \norma{\tilde \rho_o}_{\L1 (\Omega;\reali^n)}, \,
      \norma{\rho_o}_{\L\infty (\Omega;\reali^n)}, \,
      \norma{\tilde \rho_o}_{\L\infty (\Omega;\reali^n)}, \,
      \tv (\rho_o), \,
      \tv (\tilde \rho_o)
    \right\}.
  \end{displaymath}
  To make evident the dependence of $\mathcal{T}$ on the initial
  datum, introduce the space
  \begin{displaymath}
    \mathcal{Y}_R
    =
    \left\{
      \rho_o \in (\L\infty \cap \BV) (\Omega;\reali^n)  \colon
      \norma{\rho_o}_{\L1 (\Omega;\reali^n)} \leq R, \,
      \norma{\rho_o}_{\L\infty (\Omega;\reali^n)} \leq R, \,
      \tv(\rho_o) \leq R
    \right\}
  \end{displaymath}
  and slightly modify the map $\mathcal{T}$ to
  \begin{displaymath}
    \begin{array}{ccccl}
      \mathcal{T}
      & \colon
      & \mathcal{X}_R
        \times
        \mathcal{Y}_R
      & \to
      & \mathcal{X}_R
      \\
      &
      & r
        , \,
        \rho_o
      & \to
      & \rho
    \end{array}
  \end{displaymath}
  where $\rho$ solves~\eqref{eq:3}. The map $\mathcal{T}$ is a
  contraction in $r \in \mathcal{X}_R$, Lipschitz continuous in
  $\rho_o \in \mathcal{Y}_R$, when restricted to functions defined on
  each time interval $[T_k, T_{k+1}]$. In particular,
  \begin{align*}
    & \norma{
      \mathcal{T} (r, \rho (T_k)) -
      \mathcal{T} (\tilde r, \tilde\rho (T_k))}_{\L\infty ([T_k, T_{k+1}];\L1 (\Omega;\reali^n))}
    \\ \leq \
    & \dfrac{1}{2} \,
      \norma{r-\tilde r}_{\L\infty ([T_k, T_{k+1}];\L1 (\Omega;\reali^n))}
      +
      \norma{\rho (T_k) - \tilde\rho (T_k)}_{\L1 (\Omega;\reali^n)}
  \end{align*}
  by~\eqref{eq:19} and~\eqref{eq:dipLip}. Hence,
  $\norma{\rho (T_k) - \tilde\rho (T_k)}_{\L1 (\Omega; \reali^n)} \leq
  2 \norma{\rho (T_{k-1}) - \tilde\rho (T_{k-1})}_{\L1 (\Omega;
    \reali^n)}$,
  which recursively yields
  $\norma{\rho (T_k) - \tilde\rho (T_k)}_{\L1 (\Omega; \reali^n)} \leq
  2^k \, \norma{\rho_o - \tilde\rho_o}_{\L1 (\Omega; \reali^n)}$.
  The term in square brackets in the left hand side of~\eqref{eq:19}
  is uniformly bounded in $k$ by a positive constant, say,
  $A_t$. Therefore, $T_k \geq \dfrac{1}{2\,A_t\,C} + T_{k-1}$ which
  recursively yields $T_k \geq k/(2\,A_t\,C)$ and
  $k \leq 2\,A_t\,C\,T_k < 2\,A_t\,C\,t$, so that
  \begin{displaymath}
    \label{eq:22}
    \norma{\rho (t) - \tilde\rho (t)}_{\L1 (\Omega;\reali^n)}
    =
    \lim_{k\to+\infty} \norma{\rho (T_k) - \tilde\rho (T_k)}_{\L1 (\Omega; \reali^n)}
    \leq
    2^{2\, A_t \, C \, t} \; \norma{\rho_o - \tilde\rho_o}_{\L1 (\Omega;\reali^n)}
  \end{displaymath}
  completing the proof of~(\ref{item:lipid}).

  \paragraph{Stability estimate.}
  We aim to apply~\eqref{eq:14} in Lemma~\ref{lem:stab}. Exploit the
  definition
  $u^i (t,x) = V^i \left(t,x,\left(\mathcal{J}^i \rho (t)\right)
    (x)\right)$ and compute, thanks to~\textbf{(V)} and~\textbf{(J)}:
  \begin{align*}
    \norma{\nabla u^i (t)}_{\L\infty (\Omega;\reali^{N\times N})}
    \leq \
    & \norma{\nabla_x V^i}_{\L\infty (\Sigma_t;\reali^{N\times N})}
      +
      \norma{\nabla_w V^i}_{\L\infty (\Sigma_t;\reali^{N\times m})}
      \norma{\nabla_x \mathcal{J}^i \rho (t)}_{\L\infty (\Omega;\reali^{m\times N})}
    \\
    \leq \
    & \mathcal{V} \left(
      1
      +
      K \, \norma{\rho (t)}_{\L1 (\Omega;\reali^n)}
      \right)
    \\
    \leq \
    & \mathcal{V} \left(
      1
      +
      K \, \norma{\rho_o}_{\L1 (\Omega;\reali^n)}
      \right)
  \end{align*}
  and
  \begin{align*}
    &\norma{\nabla \div u^i (t)}_{\L\infty (\Omega;\reali^{N})}
    \\
    \leq \
    & \norma{\nabla_x \div V^i}_{\L\infty (\Sigma_t; \reali^N)}
    \\
    & +
      \left(\norma{\nabla_w \div V^i}_{\L\infty (\Sigma_t; \reali^m)}
      +
      \norma{\nabla_x \nabla_w V^i}_{\L\infty (\Sigma_t; \reali^{N\times m \times N})}\right)
      \norma{\nabla_x \mathcal{J}^i \rho (t)}_{\L\infty (\Omega;\reali^{m\times N})}
    \\
    & + \norma{\nabla^2_w V^i}_{\L\infty (\Sigma_t;\reali^{N \times m \times m})} \,
      \norma{\nabla_x \mathcal{J}^i \rho (t)}^2_{\L\infty (\Omega;\reali^{m\times N})}
    \\
    & +
      \norma{\nabla_w V^i}_{\L\infty (\Sigma_t; \reali^{N \times m})} \,
      \norma{\nabla_x^2 \mathcal{J}^i \rho (t)}_{\L\infty (\Omega;\reali^{m \times N \times N})}
    \\
    \leq \
    & \mathcal{V}
      \left(
      1
      +
      K \, \norma{\rho (t)}_{\L1 (\Omega;\reali^n)}
      +
      K^2 \, \norma{\rho (t)}^2_{\L1 (\Omega;\reali^n)}
      +
      \mathcal{K}\! \left(\norma{\rho (t)}_{\L1 (\Omega;\reali^n)}\right)\,
      \norma{\rho (t)}_{\L1 (\Omega;\reali^n)}
      \right)
    \\
    \leq \
    & \mathcal{V}
      \left(
      1
      +
      \norma{\rho (t)}_{\L1 (\Omega;\reali^n)}
      \left(
      K +
      K^2 \, \norma{\rho_o}_{\L1 (\Omega;\reali^n)}
      +
      \mathcal{K}\! \left(\norma{\rho_o}_{\L1 (\Omega;\reali^n)}\right)\,
      \right)
      \right)\, ,
  \end{align*}
  and the same estimates hold for each $\tilde u^i$, defined by
  $\tilde u^i (t,x) = \tilde V^i \left(t,x,\left(\mathcal{J}^i \tilde\rho
      (t)\right) (x)\right)$.

  Moreover, still by~\textbf{(V)} and~\textbf{(J)},
  \begin{align*}
    \norma{(u^i-\tilde u^i) (t)}_{\L\infty (\Omega;\reali^N)}
    = \
    & \esssup_{x \in \Omega}
      \modulo{V^i \left(t,x, \left(\mathcal{J}^i\rho (t) \right)(x)\right) -
      \tilde V^i \left(t,x, \left(\mathcal{J}^i\tilde\rho (t) \right)(x)\right)}
    \\
    \leq \
    & \esssup_{x \in \Omega}
      \modulo{V^i \left(t,x, \left(\mathcal{J}^i\rho (t) \right)(x)\right) -
      V^i \left(t,x, \left(\mathcal{J}^i\tilde\rho (t) \right)(x)\right)}
    \\
    &  +
      \esssup_{x \in \Omega}
      \modulo{V^i \left(t,x, \left(\mathcal{J}^i\tilde\rho (t) \right)(x)\right) -
      \tilde V^i \left(t,x, \left(\mathcal{J}^i\tilde\rho (t) \right)(x)\right)}
    \\
    \leq \
    & \norma{\nabla_w V^i}_{\L\infty (\Sigma_t;\reali^{N\times m})} \,
      \norma{\mathcal{J}^i\rho (t) - \mathcal{J}^i\tilde\rho (t)}_{\L\infty (\Omega;\reali^m)}
    \\
    & +
      \norma{(V^i-\tilde V^i) (t)}_{\L\infty (\Omega \times B (0, K \norma{\rho_o}_{\L1 (\Omega;\reali^n)});\reali^N)}
    \\
    \leq \
    & \mathcal{V} K
      \norma{\left(\rho  - \tilde\rho \right)(t)}_{\L1 (\Omega;\reali^n)}
      {+}
      \norma{(V^i-\tilde V^i) (t)}_{\L\infty (\Omega \times B (0, K \norma{\rho_o}_{\L1 (\Omega;\reali^n)});\reali^N)}
  \end{align*}
  and
  \begin{align*}
    & \norma{\div (u^i - \tilde u^i) (t)}_{\L\infty (\Omega;\reali)}
    \\
    \leq \
    & \esssup_{x \in \Omega} \modulo{
      \div \left(
      V^i\left(t,x,\left(\mathcal{J}^i\rho (t)\right) (x)\right)
      -
      \tilde V^i\left(t,x,\left(\mathcal{J}^i\tilde \rho (t)\right) (x)\right)
      \right)
      }
    \\
    & + \esssup_{x\in\Omega} \left|
      \nabla_w V^i \left(t,x,\left(\mathcal{J}^i\rho (t)\right) (x)\right) \cdot
      \nabla \left(\mathcal{J}^i\rho (t)\right) (x)\right.
    \\
    &\qquad\qquad\qquad \left.
      -
      \nabla_w \tilde V^i \left(t,x,\left(\mathcal{J}^i\tilde \rho (t)\right) (x)\right) \cdot
      \nabla \left(\mathcal{J}^i\tilde \rho (t)\right) (x)
      \right|
    \\
    \leq \
    & \norma{\nabla_w \div V^i}_{\L\infty (\Sigma_t;\reali^m)} \,
      \norma{
      \mathcal{J}^i\rho (t)
      -
      \mathcal{J}^i\tilde \rho (t)
      }_{\L\infty (\Omega;\reali^m)}
    \\
    & + \norma{
      \div \left( V^i - \tilde V^i \right) (t)
      }_{\L\infty (\Omega \times B (0, K \norma{\rho_o}_{\L1 (\Omega;\reali^n)});\reali)}
    \\ & + \norma{\nabla_w V^i}_{\L\infty (\Sigma_t;\reali^{N \times m})}
         \norma{\nabla \mathcal{J}^i\rho (t) -
         \nabla\mathcal{J}^i\tilde \rho (t)}_{\L\infty (\Omega;\reali^{m\times N})}
    \\
    &  +
      \norma{\nabla_w V^i (t) - \nabla_w \tilde V^i (t)}_{\L\infty (\Omega \times B (0, K \norma{\rho_o}_{\L1 (\Omega;\reali^n)});\reali^{N \times m})}
      \norma{\nabla \mathcal{J}^i\tilde \rho (t)}_{\L\infty (\Omega;\reali^{m \times N})}
    \\
    \leq \
    & \mathcal{V} \, K \,
      \norma{\left(\rho - \tilde \rho\right) (t)}_{\L1 (\Omega;\reali^n)}
      + \norma{
      \div \left( V^i - \tilde V^i \right) (t)
      }_{\L\infty (\Omega \times B (0, K \norma{\rho_o}_{\L1 (\Omega;\reali^n)});\reali)}
    \\ & + \mathcal{V} \,
         \mathcal{K}\!\left( \norma{\rho_o}_{\L1 (\Omega;\reali^n)}\right)
         \norma{\left(\rho  - \tilde \rho\right) (t)}_{\L1 (\Omega;\reali^n)}
    \\
    &  + K \, \norma{\rho_o}_{\L1 (\Omega,\reali^n)}
      \norma{\nabla_w\left( V^i  - \tilde V^i\right) (t)}_{\L\infty (\Omega \times B (0, K \norma{\rho_o}_{\L1 (\Omega;\reali^n)});\reali^{N \times m})}.
  \end{align*}
  Therefore, for all $t \in I$, by \eqref{eq:14} in
  Lemma~\ref{lem:stab}, we have
  \begin{align*}
    & \norma{\rho^i (t) - \tilde \rho^i (t)}_{\L1 (\Omega,\reali)}
    \\
    \leq \
    & \exp\left( t \, \mathcal{V} \left(
      1 + K \, \norma{\rho_o}_{\L1 (\Omega;\reali^n)}
      \right)\right)
      \biggl[
      \mathcal{O} (1) \norma{\rho_o}_{\L\infty (\Omega;\reali^n)}
      +
      \tv (\rho_o)
    \\
    & \qquad +
      t \, \mathcal{V} \,  \norma{\rho_o}_{\L1 (\Omega;\reali^n)}
      \left(
      1 + \norma{\rho_o}_{\L1 (\Omega;\reali^n)} \left(
      K + \mathcal{K}\!\left(\norma{\rho_o}_{\L1 (\Omega;\reali^n)}\right)
      + K^2 \, \norma{\rho_o}_{\L1 (\Omega;\reali^n)}
      \right)
      \right)
      \biggr]
    \\
    & \times \left(
      \mathcal{V} \, K \,
      \int_0^t \norma{\rho (s) - \tilde \rho (s)}_{\L1 (\Omega;\reali^n)} \d{s}
      +
      \int_0^t \norma{(V^i-\tilde V^i) (s)}_{\L\infty (\Omega \times B (0, K \norma{\rho_o}_{\L1 (\Omega;\reali^n)});\reali^N)} \d{s}
      \right)
    \\
    & + \norma{\rho_o}_{\L1 (\Omega;\reali^n)} \, \mathcal{V} \, \left(
      K + \mathcal{K}\!\left(\norma{\rho_o}_{\L1 (\Omega;\reali^n)}\right)
      \right) \int_0^t \norma{\rho (s) - \tilde \rho (s)}_{\L1 (\Omega;\reali^n)} \d{s}
    \\
    & \quad + \norma{\rho_o}_{\L1 (\Omega;\reali^n)} \, \int_0^t \norma{
      \div \left( V^i - \tilde V^i \right) (s)
      }_{\L\infty (\Omega \times B (0, K \norma{\rho_o}_{\L1 (\Omega;\reali^n)});\reali)} \d{s}
    \\
    & \quad +
      K \, \norma{\rho_o}_{\L1 (\Omega;\reali^n)}^2 \, \int_0^t
      \norma{\nabla_w\left( V^i  - \tilde V^i\right) (s)}_{\L\infty (\Omega \times B (0, K \norma{\rho_o}_{\L1 (\Omega;\reali^n)});\reali^{N \times m})}\d{s}
    \\
    \leq \
    & b (t) \,
      \int_o^t \norma{\rho (s) - \tilde \rho (s)}_{\L1 (\Omega;\reali^n)} \d{s}
      + c (t)
      \int_0^t \norma{V (s) - \tilde V (s)}_{\C1(\Omega \times B (0, K \norma{\rho_o}_{\L1 (\Omega;\reali^n)});\reali^{nN})} \d{s},
  \end{align*}
  where we denote
  \begin{align*}
    a (t) = \
    & \exp\left( t \, \mathcal{V} \left(
      1 + K \, \norma{\rho_o}_{\L1 (\Omega;\reali^n)}
      \right)\right)
      \biggl[
      \mathcal{O} (1) \norma{\rho_o}_{\L\infty (\Omega;\reali^n)}
      +
      \tv (\rho_o)
    \\
    & \quad +
      t \, \mathcal{V} \,  \norma{\rho_o}_{\L1 (\Omega;\reali^n)}
      \left(
      1 + \norma{\rho_o}_{\L1 (\Omega;\reali^n)} \left(
      K + \mathcal{K}\!\left(\norma{\rho_o}_{\L1 (\Omega;\reali^n)}\right)
      + K^2 \, \norma{\rho_o}_{\L1 (\Omega;\reali^n)}
      \right)
      \right)
      \biggr]
    \\
    b (t) = \
    &
      \mathcal{V} \, K \,
      a (t) +
      \norma{\rho_o}_{\L1 (\Omega;\reali^n)} \, \mathcal{V} \, \left(
      K + \mathcal{K}\!\left(\norma{\rho_o}_{\L1 (\Omega;\reali^n)}\right)
      \right)
    \\
    c (t) = \
    & a (t) + \norma{\rho_o}_{\L1 (\Omega;\reali^n)} \left(
      1 +  K \, \norma{\rho_o}_{\L1 (\Omega;\reali^n)}
      \right).
  \end{align*}
  Applying Gronwall Lemma to the resulting inequality
  \begin{align*}
    \norma{\rho (t) - \tilde \rho (t)}_{\L1 (\Omega,\reali^n)}
    \leq \
    & b (t) \,
      \int_o^t \norma{\rho (s) - \tilde \rho (s)}_{\L1 (\Omega;\reali^n)} \d{s}
    \\
    & \qquad +
      c (t)
      \int_0^t \norma{V (s) - \tilde V (s)}_{\C1(\Omega \times B (0, K \norma{\rho_o}_{\L1 (\Omega;\reali^n)});\reali^{nN})} \d{s}
  \end{align*}
  yields
  \begin{align*}
    \norma{\rho (t) - \tilde \rho (t)}_{\L1 (\Omega;\reali^n)}
    \leq \
    & c (t) \int_0^t \norma{V (s) - \tilde V (s)}_{\C1(\Omega \times B (0, K \norma{\rho_o}_{\L1 (\Omega;\reali^)});\reali^{nN})} \d{s}
    \\
    & +
      b (t) \, e^{\int_0^t b (s) \d{s}}
      \int_0^t c (s) \, e^{-\int_0^s b (\tau) \d{\tau}} \d{s}.
  \end{align*}
  Since
  $ e^{-\int_0^t b (\tau) \d{\tau}} + b (t) \int_0^t e^{-\int_0^s b
    (\tau) \d{\tau}}\d{s} \leq \frac{b (t)}{b (0)}$ we get
  \begin{equation}
    \label{eq:21}
    \!\!\!\!
    \norma{\rho (t) - \tilde \rho (t)}_{\L1 (\Omega;\reali^n)}
    \leq
    c (t) \frac{b (t)}{b (0)} e^{t \, b (t)}
    \int_0^t
    \norma{V (s) - \tilde V (s)}_{\C1(\Omega \times B (0, K \norma{\rho_o}_{\L1 (\Omega;\reali^n)});\reali^{nN})} \d{s} \!
  \end{equation}
  completing the proof.
\end{proofof}

\section{Proofs Related to Section~\ref{sec:EX}}
\label{sec:Crowd}

\begin{lemma}
  \label{lem:z}
  Let $\Omega$ and $\eta$ satisfy~$\boldsymbol{(\Omega)}$
  and~$\boldsymbol{(\eta)}$, with $r_\Omega \leq \ell_\eta/4$.  Then,
  the function $z$ defined in~\eqref{eq:z} satisfies:
  \begin{enumerate}[$(\boldsymbol{z}.\bf 1)$]
  \item There exists a $c \in \left]0, 1\right[$, depending only on
    $\Omega$ and on $\eta$, such that $z (\Omega) \subseteq [c, 1]$.
  \item $z \in \C2 (\Omega; \reali)$ and
    $\nabla z (x) = \int_\Omega \nabla \eta (x-y) \d{y}$,
    $\nabla^2 z (x) = \int_\Omega \nabla^2 \eta (x-y) \d{y}$.
  \item For all $x \in \Omega$ such that
    $B (x,\ell_\eta) \subseteq \Omega$, $z (x) = 1$.
  \end{enumerate}
\end{lemma}

\begin{proof}
  Consider first~$(\boldsymbol{z}.\bf 1)$. For all $x \in \Omega$ such
  that $B (x, \ell_\eta/2) \subseteq \Omega$, we have
  \begin{displaymath}
    z (x)
    =
    \int_\Omega \eta (x-y) \d{y}
    \geq
    \int_{B (x,\ell_\eta/2)} \eta (x-y) \d{y}
    \geq
    \int_{B (x,r_\Omega)} \eta (x-y) \d{y}
    =
    \int_{B (0,r_\Omega)} \eta (-y) \d{y} \,.
  \end{displaymath}
  If on the other hand $B (x, \ell_\eta/2)$ is not contained in
  $\Omega$, then there exists a
  $\xi \in B (x, \ell_\eta/2) \cap \partial\Omega$.  Call $x_\xi$ a
  point such that $\xi \in \partial B (x_\xi, r_\Omega)$ and
  $B(x_\xi, r_\Omega) \subseteq \Omega$, which exists by the interior
  sphere condition, ensured by~\textbf{($\boldsymbol\eta$)}. Then, for
  all $y \in B (x_\xi,r_\Omega)$, we have
  \begin{displaymath}
    \norma{y-x}
    \leq
    \norma{y-x_\xi} +\norma{x_\xi-\xi} + \norma{\xi-x}
    \leq
    2\, r_\Omega + \frac{1}{2}\, \ell_\eta
    \leq \ell_\eta
  \end{displaymath}
  showing that $B (x_\xi, r_\Omega) \subseteq B (x, \ell_\eta)$, so
  that $B (x_\xi - x, r_\Omega) \subseteq B (0, \ell_\eta)$ and
  \begin{displaymath}
    z (x)
    =
    \int_\Omega\eta (x-y) \d{y}
    \geq
    \int_{B (x_\xi, r_\Omega)} \eta (x-y) \d{y}
    =
    \int_{B (x_\xi-x, r_\Omega)} \eta (-y) \d{y}\,.
  \end{displaymath}
  In both cases, applying {Weiestra\ss} Theorem to the continuous map
  $\alpha \to \int_{B (\alpha, r_\Omega)} \eta (-y) \d{y}$, for all
  $x \in \Omega$ we obtain
  \begin{align*}
    z (x)
    \geq \
    & \inf_{\alpha \colon B (\alpha, r_\Omega) \subseteq  B (0,\ell_\eta)}
      \int_{B (\alpha, r_\Omega)} \eta (-y) \d{y}
    \\
    = \
    & \inf_{\alpha \in B (0,\ell_\eta - r_\Omega)}
      \int_{B (\alpha, r_\Omega)} \eta (-y) \d{y}
      \ =
      \min_{\alpha \in B (0,\ell_\eta - r_\Omega)}
      \int_{B (\alpha, r_\Omega)} \eta (-y) \d{y} \,.
  \end{align*}
  Define now
  $c = \min_{\alpha \in B (0,\ell_\eta - r_\Omega)} \int_{B (\alpha,
    r_\Omega)} \eta (-y) \d{y}$:
  note that this quantity is strictly positive and strictly less than
  $1$ by~\textbf{($\boldsymbol\eta$)}.  The proof
  of~$(\boldsymbol{z}.\bf 1)$ is completed.

  The proof of~$(\boldsymbol{z}.\bf 2)$ follows noting that
  $z = \caratt{\Omega} * \eta$, applying the usual properties of the
  convolution:
  $\nabla z = \nabla (\caratt{\Omega} * \eta) = \caratt{\Omega} *
  \nabla z$ and a similar computation yields $\nabla^2 z$.

  The property~$(\boldsymbol{z}.\bf 3)$ is immediate.
\end{proof}

\begin{proofof}{Lemma~\ref{lem:Omega}}
  The $\C2$ regularity follows from the standard properties of the
  convolution product and from Lemma~\ref{lem:z}. The lower and upper
  bounds on $\rho \sOmega \eta$ are immediate. For the latter one, for
  instance, $    (\rho \sOmega \eta) (x)
    \leq
    \dfrac{1}{z (x)} \,
    \left(\esssup_{B (x,\ell_\eta) \cap \Omega} \rho\right) \,
    \int_\Omega \eta (x-y) \d{y}
    =
    \esssup_{B (x,\ell_\eta) \cap \Omega} \rho$,
  % \begin{displaymath}
  %   % (\rho \sOmega \eta) (x)
  %   % & \geq
  %   % & \dfrac{1}{z (x)} \,
  %   % \left(\essinf_{B (x,\ell_\eta) \cap \Omega} \rho\right) \,
  %   % \int_\Omega \eta (x-y) \d{y}
  %   % =
  %   % \essinf_{B (x,\ell_\eta) \cap \Omega} \rho \,,
  %   % \\
  %   (\rho \sOmega \eta) (x)
  %   \leq
  %   \dfrac{1}{z (x)} \,
  %   \left(\esssup_{B (x,\ell_\eta) \cap \Omega} \rho\right) \,
  %   \int_\Omega \eta (x-y) \d{y}
  %   =
  %   \esssup_{B (x,\ell_\eta) \cap \Omega} \rho \,,
  % \end{displaymath}
  completing the proof.
\end{proofof}

\begin{proofof}{Lemma~\ref{lem:Corridor}}
  With reference to the notation in Section~\ref{sec:MR}, set $N=2$,
  $n=1$, $m=3$.  Call $\boldsymbol{i}$, respectively $\boldsymbol{j}$,
  a unit vector directed along the $x_1$, respectively $x_2$,
  axis. Define
  \begin{displaymath}
    V (t, x, A)
    =
    v (A_1)
    \left(
      w (x)
      -
      \beta \,
      \frac{A_2 \, \boldsymbol{i} + A_3 \, \boldsymbol{j}}{\sqrt{1 + {A_2}^2 + {A_3}^2}}
    \right)
    \quad \mbox{ with }\quad
    \mathcal{J} (\rho) =
    \left[%
      \begin{array}{c}
        \rho \sOmega \eta_1
        \\
        \partial_1 (\rho \sOmega \eta_2)
        \\
        \partial_2 (\rho \sOmega \eta_2)
      \end{array}
    \right] \,.
  \end{displaymath}
  Clearly, $V \in \C2 (\Omega \times \reali^3; \reali^2)$. The $\C2$
  boundedness of $V$ follows from that of $v$, from that of $w$, from
  that of the map
  $(A_2, A_3) \to \frac{A_2 \, \boldsymbol{i} + A_3 \,
    \boldsymbol{j}}{\sqrt{1 + {A_2}^2 + {A_3}^2}}$ and from the
  compactness of $\overline{\Omega}$. Hence, \textbf{(V)} holds.

  Concerning~\textbf{(J)}, the $\C2$ regularity follows
  from~\textbf{($\boldsymbol\eta$)}, from Lemma~\ref{lem:z} and from
  the assumption $\eta_2 \in \C3$. To prove~\textbf{(J.1)}, with the
  notation in Lemma~\ref{lem:z}, consider the different components of
  $\mathcal{J}$ separately. Recall that $z = \caratt{\Omega} * \eta$
  and write the first component of $\mathcal{J} \rho$ as
  $ \rho \sOmega \eta_1 = \bigl((\rho \, \caratt{\Omega}) *
  \eta\bigr)/z$:
  \begin{align*}
    \norma{\rho \sOmega \eta_1}_{\L\infty (\Omega; \reali)}
    \leq \
    & \dfrac{\norma{\eta_1}_{\L\infty (\reali^2; \reali)}}c \,
      \norma{\rho}_{\L1 (\Omega;\reali)}
    \\
    \nabla (\rho \sOmega \eta_1)
    = \
    & \dfrac{1}{z} \,
      \left((\rho \caratt{\Omega}) * \nabla \eta_1\right)
      -
      \dfrac{\caratt{\Omega} * \nabla \eta_1}{z^2} \,
      \left((\rho \caratt{\Omega}) * \eta_1\right)
    \\
    \norma{\nabla (\rho \sOmega \eta_1)}_{\L\infty (\Omega; \reali^2)}
    \leq \
    & \left(
      \dfrac{\norma{\nabla\eta_1}_{\L\infty (\reali^2; \reali^2)}}{c}
      +
      \dfrac{
      \norma{\nabla\eta_1}_{\L1 (\reali^2; \reali^2)} \,
      \norma{\eta_1}_{\L\infty (\reali^2; \reali)}
      }{c^2}
      \right)
      \norma{\rho}_{\L1 (\Omega;\reali)}
    \\
    \nabla^2 (\rho \sOmega \eta_1)
    = \
    & \frac{1}{z} \left((\rho  \caratt{\Omega}) * \nabla^2 \eta_1\right)
      - 2 \, \frac{\caratt{\Omega} * \nabla \eta_1}{z^2}
      \left((\rho  \caratt{\Omega}) * \nabla \eta_1\right)
    \\
    & -
      \left((\rho  \caratt{\Omega}) * \eta_1\right)
      \left(
      \frac{\caratt{\Omega} * \nabla^2 \eta_1}{z^2}
      - \frac{2}{z^3}
      \left(\caratt{\Omega} * \nabla \eta_1 \right)
      \otimes
      \left(\caratt{\Omega} * \nabla \eta_1 \right)
      \right)
    \\
    \norma{\nabla^2 (\rho \sOmega \eta_1)}_{\L\infty (\Omega; \reali^{2\times2})}
    \leq
    &  \left[
      \frac{\norma{\nabla^2 \eta_1}_{\L\infty (\reali^2; \reali^{2\times 2})}}{c}
      +
      \frac{
      \norma{\nabla\eta_1}_{\L1 (\reali^2; \reali^2)} \,
      \norma{\nabla \eta_1}_{\L\infty (\reali^2; \reali^2)}
      }{c^2}\right.
    \\
    & \left.
      {+}
      \frac{\norma{\eta_1}_{\L\infty (\reali^2;\reali)}}
      {c^2} \!
      \left( \!
      \norma{\nabla^2 \eta_1}_{\L1 (\reali^2; \reali^{2\times 2})}
      \!{+}
      \frac{2 \norma{\nabla \eta_1}^2_{\L1 (\reali^2; \reali^2)}}{c}
      \! \right) \!
      \right] \!\!
      \norma{\rho}_{\L1 (\Omega;\reali)}.
  \end{align*}
  The estimates of $\partial_j (\rho \sOmega \eta_2)$ and
  $\nabla \partial_j (\rho \sOmega \eta_2)$, for $j=1,2$, are entirely
  analogous. We only check
  \begin{align*}
    \nabla^2 \partial_j (\rho \sOmega \eta_2)
    = \
    & \frac{1}{z} \left((\rho  \caratt{\Omega}) * \nabla^2 \partial_j \eta_2\right)
      -  \frac{\caratt{\Omega} * \partial_j \eta_2}{z^2}
      \left((\rho  \caratt{\Omega}) * \nabla^2  \eta_2\right)
    \\
    & - 2 \, \frac{\caratt{\Omega} * \nabla \eta_2}{z^2}
      \left((\rho  \caratt{\Omega}) * \nabla\partial_j \eta_2\right)
      -2 \, \frac{\caratt{\Omega} * \nabla \partial_j \eta_2}{z^2}
      \left((\rho  \caratt{\Omega}) * \nabla \eta_2\right)
    \\
    & + \frac{4}{z^3}
      \left(\caratt{\Omega} * \nabla \eta_2\right)
      \left(\caratt{\Omega} * \partial_j \eta_2\right)
      \left((\rho  \caratt{\Omega}) * \nabla \eta_2\right)
    \\
    & - \frac{\caratt{\Omega} * \nabla^2 \eta_2}{z^2} \!
      \big((\rho  \caratt{\Omega}) * \partial_j \eta_2\big)
      + \frac{2}{z^3}
      \big(\caratt{\Omega} * \nabla^2 \eta_2\big)
      \big(\caratt{\Omega} * \partial_j \eta_2\big)
      \big((\rho  \caratt{\Omega}) *  \eta_2\big)
    \\
    & -\frac{\caratt{\Omega} * \nabla^2 \partial_j \eta_2}{z^2}
      \big((\rho  \caratt{\Omega}) *  \eta_2\big)
      + \frac{2}{z^3} \!
      \big(\caratt{\Omega} * \nabla \eta_2\big) {\otimes}
      \big(\caratt{\Omega} * \nabla \eta_2\big)
      \big((\rho  \caratt{\Omega}) *  \partial_j \eta_2\big)
    \\
    & - 6 \, \frac{\caratt{\Omega} * \partial_j \eta_2}{z^4}
      \left(\caratt{\Omega} * \nabla \eta_2\right) \otimes
      \left(\caratt{\Omega} * \nabla \eta_2\right)
      \left((\rho  \caratt{\Omega}) *   \eta_2\right)
    \\
    & + \frac{4}{z^3}
      \left(\caratt{\Omega} * \nabla \eta_2\right)
      \left(\caratt{\Omega} * \nabla \partial_j\eta_2\right)
      \left((\rho  \caratt{\Omega}) *   \eta_2\right)
  \end{align*}
  \begin{align*}
    & \norma{\nabla^2 \partial_j (\rho \sOmega \eta_2)}_{\L\infty (\Omega; \reali^{2\times2})}
    \\
    \leq \
    &\left(
      \frac{\norma{\nabla^2 \partial_j \eta_2}_{\L\infty (\reali^2; \reali^{2 \times 2})}}{c}
      +\frac{\norma{\partial_j \eta_2}_{\L1 (\reali^2; \reali)} \,
      \norma{\nabla^2 \eta_2}_{\L\infty (\reali^2; \reali^{2 \times 2})}}
      {c^2}
      \right.
    \\
    & + 2 \, \frac{\norma{\nabla  \eta_2}_{\L1 (\reali^2; \reali^2)} \,
      \norma{\nabla \partial_j \eta_2}_{\L\infty (\reali^2; \reali^2 )}}
      {c^2}
      + 2 \, \frac{\norma{\nabla  \partial_j \eta_2}_{\L1 (\reali^2; \reali^2)} \,
      \norma{\nabla \eta_2}_{\L\infty (\reali^2; \reali^2 )}}
      {c^2}
    \\
    & + 4 \, \frac{\norma{\nabla  \eta_2}_{\L1 (\reali^2; \reali^2)} \,
      \norma{\partial_j \eta_2}_{\L1 (\reali^2; \reali)} \,
      \norma{\nabla \eta_2}_{\L\infty (\reali^2; \reali^2)}}
      {c^3}
      + \frac{\norma{\nabla^2 \eta_2}_{\L1 (\reali^2; \reali^{2 \times 2})} \,
      \norma{\partial_j \eta_2}_{\L\infty (\reali^2; \reali)}}
      {c^2}
    \\
    & +2 \, \frac{\norma{\nabla^2 \eta_2}_{\L1 (\reali^2; \reali^{2 \times 2})} \,
      \norma{\partial_j \eta_2}_{\L1 (\reali^2; \reali)} \,
      \norma{\eta_2}_{\L\infty (\reali^2; \reali)}}
      {c^3}
      + \frac{\norma{\nabla^2 \partial_j \eta_2}_{\L1 (\reali^2; \reali^{2 \times 2})} \,
      \norma{\eta_2}_{\L\infty (\reali^2; \reali)}}{c^2}
    \\
    & + 2 \, \frac{\norma{\nabla  \eta_2}^2_{\L1 (\reali^2; \reali^2)} \,
      \norma{\partial_j \eta_2}_{\L\infty (\reali^2; \reali)}}{c^3}
      + 6 \, \frac{\norma{\partial_j \eta_2}_{\L1 (\reali^2; \reali)} \,
      \norma{\nabla^2 \eta_2}^2_{\L1 (\reali^2; \reali^2)} \,
      \norma{\eta_2}_{\L\infty (\reali^2; \reali)}}{c^4}
    \\
    &
      \left.
      + 4 \, \frac{\norma{\nabla\eta_2}_{\L1 (\reali^2; \reali^2)} \,
      \norma{\nabla \partial_j \eta_2}_{\L1 (\reali^2; \reali^2)} \,
      \norma{\eta_2}_{\L\infty (\reali^2; \reali)}}{c^3}
      \right)
      \norma{\rho}_{\L1 (\Omega;\reali)}.
  \end{align*}
  Finally, \textbf{(J.2)} is now immediate thanks to the linearity of
  $\mathcal{J}$.
\end{proofof}

\begin{proofof}{Lemma~\ref{lem:TwoPop}}
  Note that~\eqref{eq:24} fits into~\eqref{eq:1} setting $N=2$, $n=2$,
  $m=10$ and
  \begin{equation}
    \label{eq:17}
    \begin{array}{@{}c@{}}
      \begin{array}{@{}r@{\,}c@{\,}l}
        V^1 (t,x,A)
        & =
        & v^1 (A_1) \left(
          w^1 (x)
          -
          \frac{\beta_{11} (A_3 \boldsymbol{i} + A_4 \boldsymbol{j})}{\sqrt{1+{A_3}^2+{A_4}^2}}
          -
          \frac{\beta_{12} (A_5 \boldsymbol{i} + A_6 \boldsymbol{j})}{\sqrt{1+{A_5}^2+{A_6}^2}}
          \right) \,,
        \\
        V^2 (t,x,A)
        & =
        & v^2 (A_2) \left(
          w^2 (x)
          -
          \frac{\beta_{21} (A_7 \boldsymbol{i} + A_8 \boldsymbol{j})}{\sqrt{1+{A_7}^2+{A_8}^2}}
          -
          \frac{\beta_{22} (A_9 \boldsymbol{i} + A_{10} \boldsymbol{j})}{\sqrt{1+{A_9}^2+{A_{10}}^2}}
          \right) \,,
      \end{array}
      \\
      \mathcal{J} (\rho)_1 = (\rho_1+\rho_2) \sOmega \eta^{11}_1
      \,,\qquad
      \mathcal{J} (\rho)_{3,4} = \nabla_x (\rho_1 \sOmega \eta^{11}_2)
      \,,\qquad
      \mathcal{J} (\rho)_{5,6} = \nabla_x (\rho_1 \sOmega \eta^{12}_2)\,,
      \\
      \mathcal{J} (\rho)_2 = (\rho_1+\rho_2) \sOmega \eta^{22}_1
      \,,\qquad
      \mathcal{J} (\rho)_{7,8} = \nabla_x (\rho_1 \sOmega \eta^{21}_2)
      \,,\qquad
      \mathcal{J} (\rho)_{9,10} = \nabla_x (\rho_1 \sOmega \eta^{22}_2)\,,
    \end{array}
  \end{equation}
  where $\nabla_x =[\partial_1 \quad \partial_2]$.  The same
  computations as in the proof of Lemma~\ref{lem:Corridor} show
  that~\textbf{(V)} and~\textbf{(J)} hold, completing the proof.
\end{proofof}

\noindent\textbf{Acknowledgement:} The second author was supported at the University of Brescia by the MATHTECH project funded by CNR and INdAM. Both authors acknowledge the PRIN 2015 project \emph{Hyperbolic Systems of Conservation Laws
  and Fluid Dynamics: Analysis and Applications} and the INDAM--GNAMPA
2017 project \emph{Conservation Laws: from Theory to Technology}.

\let\OLDthebibliography\thebibliography
\renewcommand\thebibliography[1]{
  \OLDthebibliography{#1}
  \setlength{\parskip}{0pt}
  \setlength{\itemsep}{1pt plus 0.3ex}
}

{
  \small

  \bibliography{ColomboRossi}

  \bibliographystyle{abbrv}

}

\end{document}